\documentclass[12pt,reqno]{amsart}
\usepackage{amsmath, amsfonts, amssymb, amsthm, amscd, amsbsy}
\usepackage{fancyhdr}
\usepackage[usenames,dvipsnames,svgnames,x11names,hyperref]{xcolor}
\usepackage{geometry}
\usepackage{graphicx}
\usepackage[pagebackref]{hyperref}%
\usepackage{amsmath}%
\usepackage{amsfonts}%
\usepackage{amssymb}
\usepackage{enumerate}

\providecommand{\U}[1]{\protect\rule{.1in}{.1in}}
\hypersetup{backref=true,
pagebackref=true,
hyperindex=true,
colorlinks=true,
breaklinks=true,
urlcolor=NavyBlue,
linkcolor=Fuchsia,
bookmarks=true,
bookmarksopen=false,
filecolor=black,
citecolor=ForestGreen,
linkbordercolor=red
}

\allowdisplaybreaks

\newtheorem{theorem}{Theorem}[section]
\theoremstyle{plain}

\newtheorem{corollary}{Corollary}[section]

\newtheorem{lemma}{Lemma}[section]

\newtheorem{proposition}{Proposition}
\newtheorem{remark}{Remark}

\numberwithin{equation}{section}

\def\vint{\mathop{\mathchoice%
          {\setbox0\hbox{$\displaystyle\intop$}\kern 0.22\wd0%
           \vcenter{\hrule width 0.6\wd0}\kern -0.82\wd0}%
          {\setbox0\hbox{$\textstyle\intop$}\kern 0.2\wd0%
           \vcenter{\hrule width 0.6\wd0}\kern -0.8\wd0}%
          {\setbox0\hbox{$\scriptstyle\intop$}\kern 0.2\wd0%
           \vcenter{\hrule width 0.6\wd0}\kern -0.8\wd0}%
          {\setbox0\hbox{$\scriptscriptstyle\intop$}\kern 0.2\wd0%
           \vcenter{\hrule width 0.6\wd0}\kern -0.8\wd0}}%
          \mathopen{}\int}

\begin{document}

\title[Optimal  bounds for   stability of HLS and Sobolev inequalities]{Optimal  stability of Hardy-Littlewood-Sobolev and Sobolev inequalities of arbitrary orders with dimension-dependent constants}
\author{Lu Chen}
\address[Lu Chen]{Key Laboratory of Algebraic Lie Theory and Analysis of Ministry of Education, School of Mathematics and Statistics, Beijing Institute of Technology, Beijing
100081, PR China}
\email{chenlu5818804@163.com}

\author{Guozhen Lu}
\address[Guozhen Lu]{Department of Mathematics, University of Connecticut, Storrs, CT 06269, USA}
\email{guozhen.lu@uconn.edu}

\author{Hanli Tang}
\address[Hanli Tang]{Laboratory of Mathematics and Complex Systems (Ministry of Education), School of Mathematical Sciences, Beijing Normal University, Beijing, 100875, China}
\email{hltang@bnu.edu.cn}
\address{}
\thanks{The first author was partly supported by the National Key Research and Development Program (No.
2022YFA1006900) and National Natural Science Foundation of China (No. 12271027). The second author was partly supported by a Simons grant and a Simons Fellowship from the Simons Foundation.  The third author
was partly supported by National Key Research and Development Program (No. 2020YFA0712900) and the Fundamental Research Funds for the Central Universities(2233300008)}


\begin{abstract}Recently,
Dolbeault-Esteban-Figalli-Frank-Loss \cite{DEFFL}   established the optimal stability of the first-order $L^2$-Sobolev inequality  with dimension-dependent constant. Subsequently,   Chen-Lu-Tang \cite{CLT2} obtained  the optimal stability for the $L^2$ fractional Sobolev inequality of order $s$ when $0<s<1$ with dimension-dependent  and order-dependent constants.     The purpose of this paper is to consider the remaining unsolved case $1<s<\frac{n}{2}$. Our strategy is to  first establish the optimal stability for the HLS inequality directly without using the stability of the Sobolev inequality. The main difficulty lies in establishing the optimal local stability of HLS inequality  when $1<s<\frac{n}{2}$.
The loss of the Hilbert structure of the distance appearing in the stability of the HLS inequality brings much challenge in establishing the desired stability. To achieve our goal, we develop a new strategy based on the $H^{-s}-$decomposition instead of $L^{\frac{2n}{n+2s}}-$decomposition to obtain the local stability of the HLS inequality with $L^{\frac{2n}{n+2s}}-$distance.  However,  new difficulties arise   to deduce the global stability from the local stability  because of the non-uniqueness  and non-continuity of $\|r\|_{\frac{2n}{n+2s}}$ for the rearrangement flow.  This is overcome by the norm comparison theorem for $\|r\|_{\frac{2n}{n+2s}}$ and ``new continuity" theorem for the rearrangement flow  (see Lemma \ref{comparable}, Lemma \ref{continuity} and Lemma \ref{nonlocal stability}).  

As an important application of the optimal stability of the HLS inequality together with the duality theory of the  stability developed initially by Carlen \cite{Car} and further improved in \cite{CLT1}, we  deduce the optimal stability of  the $L^2$-Sobolev inequality of order s when $1\le s<\frac{n}{2}$ and the non-Hilbertian $L^{\frac{2n}{n+2s}}$-Sobolev inequality with the dimension-dependent constants. As another application,  we can derive the optimal stability of Beckner's \cite{Beckner1} restrictive Sobolev inequality on the flat sub-manifold $\mathbb{R}^{n-1}$ and the sphere $\mathbb{S}^{n-1}$ with dimension-dependent constants. 

\end{abstract}

\maketitle
\section{Introduction}
\vskip0.3cm
Lieb's sharp form \cite{Lieb} of the Hardy-Littlewood-Sobolev inequality in $\mathbb{R}^n$  for $0<s<\frac {n}{2}$  states
\begin{equation}
	\label{eq-hls}
	\| G\|_{\frac{2n}{n+2s}}^2\geq \mathcal S_{s,n}\left \|(-\Delta)^{-s/2} G \right\|_2^2
	\qquad\text{for all}\ G\in L^{\frac{2n}{n+2s}}(\mathbb{R}^n)
\end{equation}
with
\begin{equation}
	\label{eq:sobconst}
	\mathcal S_{s,n} = (4\pi)^s \ \frac{\Gamma(\frac{n+2s}{2})}{\Gamma(\frac{n-2s}{2})} \left( \frac{\Gamma(\frac n2)}{\Gamma(n)} \right)^{2s/n}
	= \frac{\Gamma(\frac{n+2s}{2})}{\Gamma(\frac{n-2s}{2})} \ |\mathbb{S}^n|^{2s/n} \,
\end{equation}
being the sharp constant. Furthermore, he also showed that the equality of HLS inequality \eqref{eq-hls} holds if and only if
$$G\in \mathcal{M}_{HLS}:=\left\{cF(\frac{\cdot-x_0}{a}): c\in \mathbb{R},\  x_0\in \mathbb{R}^n,\  a>0 \right\},$$
where $F(x)=(1+|x|^2)^{-\frac{n+2s}{2}}$.

Lieb in \cite{Lieb} also proved the sharp Sobolev inequality in $\mathbb{R}^n$ as an equivalent reformulation of the sharp Hardy-Littlewood-Sobolev inequality in the case of conformal index. Specifically, he proved for $0<s<\frac {n}{2}$ there holds
\begin{equation}
	\label{eq-sob}
	\left\|(-\Delta)^{s/2} U \right\|_2^2 \geq \mathcal S_{s,n} \| U\|_{\frac{2n}{n-2s}}^2
	\qquad\text{for all}\ U\in \dot H^s(\mathbb{R}^n)
\end{equation}
where $\dot H^s(\mathbb{R}^n)$ denotes the completion of $C_{c}^{\infty}(\mathbb{R}^n)$ under the norm $\big(\int_{\mathbb{R}^n}|(-\Delta)^{\frac{s}{2}}U|^2dx\big)^{\frac{1}{2}}$.
The sharp constant $\mathcal S_{s,n}$ of inequality (\ref{eq-sob}) has been computed first by Rosen \cite{Ro} in the case $s=1$, $n=3$ and then independently by Aubin \cite{Au} and Talenti \cite{Ta} in the case $s=1$. Furthermore, Lieb also showed that the equality of Sobolev inequality \eqref{eq-sob} holds if and only if
$$U\in \mathcal{M}_s:=\left\{cV(\frac{\cdot-x_0}{a}): c\in \mathbb{R},\  x_0\in \mathbb{R}^n,\  a>0 \right\},$$
where $V(x)=(1+|x|^2)^{-\frac{n-2s}{2}}$.
\vskip0.1cm

The stability of Sobolev inequality started from Brezis and Lieb. In \cite{BrLi} they asked if the following refined first order Sobolev inequality ($s=1$ in (\ref{eq-sob}))
holds for some distance function $d$:
$$\left\|(-\Delta)^{1/2} U \right\|_2^2 - \mathcal S_{1,n} \| U\|_{\frac{2n}{n-2}}^2\geq c d^{2}(U, \mathcal{M}_1).$$
This question was answered affirmatively  in a pioneering work by Bianchi and Egnell \cite{BiEg}, in the case $s=2$ by the second author and Wei \cite{LuWe} and in the case of any positive even integer $s<n/2$ by
Bartsch, Weth and Willem \cite{BaWeWi}. In 2013, Chen, Frank and Weth \cite{ChFrWe} established the stability of Sobolev inequality for all $0<s<n/2$. They proved that
\begin{equation}\label{Sob sta ine}
\left\|(-\Delta)^{s/2} U \right\|_2^2 - \mathcal S_{s,n} \| U\|_{\frac{2n}{n-2s}}^2\geq C_{n,s} d^{2}(U, \mathcal{M}_s),
\end{equation}
for all $U\in \dot H^s(\mathbb{R}^n)$, where $C_{n,s}>0$ and $d(U,\mathcal{M}_s)=\min\{\|(-\Delta)^{s/2}(U-\phi)\|_{L^2}:\phi \in \mathcal{M}_s\}$.
\vskip0.1cm

Although the fractional Sobolev inequality is equivalent to the HLS inequality of diagonal case, their stabilities are not fully equivalent. Carlen \cite{Car} developed a duality theory of  stability  based on the quantitative convexity and deduced the following stability of HLS inequality from the stability of fractional Sobolev inequality: There exists some positive constant $C_0(n,s)$ such that for any $G\in L^{\frac{2n}{n+2s}}(\mathbb{R}^n)$, there holds
$$\|G\|^2_{\frac{2n}{n+2s}}-S_{n,s}\|(-\Delta)^{-s/2}G\|_2^2\geq C_0(n,s)\inf\limits_{H\in \mathcal{M}_{HLS}}\|G-H\|^2_{\frac{2n}{n+2s}}.$$
For the study of stability of other kinds of functional
and geometric inequalities such as $L^p$-Sobolev inequality, Log-Sobolev inequality, Moser-Onofri inequality, Sobolev trace inequality, Gagliardo-Nirenberg-Sobolev inequality, Caffarelli-Kohn-Nirenberg identities, Brunn-Minkowski inequality etc., we refer the interested readers to the papers \cite{CFMP}, \cite{FiNe}, \cite{FiZh}, \cite{CLT}, \cite{CLT2}, \cite{CLT1}, \cite{BDNS}, \cite{CaF}, \cite{CFLL},
\cite{DFLL}, \cite{JF1}, \cite{JF2} and references therein.
\vskip0.1cm

So far the stability of Sobolev inequality is proved by establishing the local stability
of Sobolev inequality based on the spectrum analysis of elliptic or high order elliptic operator and using the Lions' concentration compactness technique to obtain global stability of Sobolev inequality. However, this method
does not tell us more information about the constant $C_{n,s}$ except it is positive.
Recently, Dolbeault, Esteban, Figalli, Frank and Loss in \cite{DEFFL} obtained for the first time the optimal asymptotic behavior for the lower bound $C_{n,s}$ when $s=1$. In fact, they proved

\vskip0.3cm
\textbf{Theorem A. }
\textit{There is an explicit constant} $\beta>0$ \textit{such that for all} $n\geq 3$ \textit{and for all} $U\in \dot H^1(\mathbb{R}^n)$, \textit{there holds}
$$\left\|\nabla U \right\|_2^2 - \mathcal S_{1,n} \| U\|_{\frac{2n}{n-2}}^2\geq \frac{\beta}{n}d^2(U,\mathcal{M}_1).$$
\vskip0.3cm

As an application, they also established the stability for Gaussian log-Sobolev inequality. The strategy adopted  by the authors of \cite{DEFFL} is to establish the local stability with explicit and optimal remainder term in the first step, where spherical harmonics and the technique of complicated orthogonality is used. In the second step, they prove the stability for nonnegative functions away from the minimizer set $\mathcal{M}_1$ by the competing symmetries method developed in \cite{CaLo} and a rearrangement flow of the $L^2$ integral of the gradient. At last they finish their proof by a concavity argument.
\vskip0.1cm

Inspired by the work of Dolbeault, Esteban, Figalli, Frank and Loss \cite{DEFFL}, the current authors in \cite{CLT1}  extended  their  result to general $0<s<\frac{n}{2}$ with the explicit lower bounds. Furthermore, the authors obtained \cite{CLT2}   the optimal asymptotic lower bounds when the dimension $n\rightarrow \infty$ for $0<s<1$ and when $s\rightarrow 0$ for all dimension $n$. However, it should be noted that the explicit lower bound obtained in \cite{CLT1} is not optimal in the asymptotic sense when $n\rightarrow +\infty$ for fractional order $1<s<\frac{n}{2}$. Therefore, the optimal asymptotic lower bound
when $n\rightarrow +\infty$ for fractional order $1<s<\frac{n}{2}$ remains to be unknown.
\vskip0.1cm

 We now recall briefly how we accomplish our goal in \cite{CLT1, CLT2}. More specifically, we first establish the local stability of Sobolev inequality for nonnegative function with explicit lower bound for general $0<s<\frac{n}{2}$, or with optimal lower bounds for $0<s<1$. Then by a duality method developed by Carlen \cite{Car}, we obtain the local stability of HLS inequality with  explicit  lower bounds for all $0<s<\frac{n}{2}$ and optimal lower bounds  for $0<s<1$. Then, using the competing symmetries method, a rearrangement flow of HLS integral which is effective for all $0<s<n/2$ and a concavity argument, we can establish the stability of HLS inequality with explicit  lower bounds for $0<s<\frac{n}{2}$ and optimal lower bounds when $n\to \infty$ for $0<s<1$. Finally, by the duality method again, we can get the global stability of the Sobolev inequality with explicit lower bounds for $0<s<\frac{n}{2}$ or optimal lower bounds for $0<s<1$.
\vskip0.1cm

In this paper, our goal is to establish both the stability of HLS inequality and the stability of the  Sobolev inequaly with the optimal asymptotic lower bound for all $1\leq s<n/2$.
The first result of this paper is to establish the optimal
asymptotic lower bound for stability of the HLS inequality when $n\rightarrow \infty$ for $1\leq s<n/2$.

\vskip0.3cm
\begin{theorem}\label{sta of hls}
For any fixed $s\in[1,n/2)$, there exists a positive constant $\beta_s$ such that for any $g\in L^{\frac{2n}{n+2s}}(\mathbb{R}^n)$, there holds
$$\|g\|^2_{\frac{2n}{n+2s}}-\mathcal{S}_{s,n}\|(-\Delta)^{-s/2}g\|_2^2\geq \frac{\beta_s}{n}\inf_{h\in \mathcal{M}_{HLS}}\|g-h\|^2_{\frac{2n}{n+2s}}.$$
\end{theorem}

\begin{remark}
 We only prove that Theorem \ref{sta of hls} holds when $n\geq N_s$. However, when $n\leq N_s$, we can apply the lower bound estimate for stability of HLS obtained in \cite{CLT1} to get the lower bound $\beta_{s}$ independent of $n$.
 \end{remark}
Let $(-\Delta)^{-\frac{s}{2}}g=u$, the above optimal stability of HLS inequality can lead to the stability of $L^{\frac{2n}{n+2s}}$-Sobolev inequality for high-order derivatives.

\begin{corollary}
For any fixed $s\in[1,n/2)$, there exists a positive constant $\beta_s$ such that for any $u\in W^{s, \frac{2n}{n+2s}}(\mathbb{R}^n)$, there holds
$$\|(-\Delta)^{\frac{s}{2}}u\|^2_{\frac{2n}{n+2s}}-\mathcal{S}_{s,n}\|u\|_2^2\geq \frac{\beta_s}{n}\inf_{h\in \widetilde{\mathcal{M}_{HLS}}}\|(-\Delta)^{\frac{s}{2}}(u-h)\|^2_{\frac{2n}{n+2s}},$$
where $\widetilde{\mathcal{M}_{HLS}}$ denotes all extremal functions of $L^{\frac{2n}{n+2s}}$-Sobolev inequality for high-order derivatives.
\end{corollary}

\begin{remark}
The stability of $L^p$-Sobolev inequality for first order derivative has been established by Figalli-Neumayer in \cite{FiNe} and Figalli-Zhang in \cite{FiZh} for $1<p<n$. However, in the non-Hilbertian case when $p\not = 2$, the optimal and even the explicit constant of the lower bound for the stability of the Sobolev inequality is not known. Our result provides an optimal lower bound in the special case $p=\frac{2n}{n+2s}$    for the $L^{\frac{2n}{n+2s}}$-Sobolev inequality involving high-order derivatives.
\end{remark}

By the duality method, we can set up the optimal
asymptotic lower bound for the stability of Sobolev inequality when $n\rightarrow \infty$ for all $1\leq s<n/2$.
\begin{theorem}\label{sta of so}
For any fixed $s\in[1,n/2)$, there exists a positive constant $\beta_s$ such that  for any $f\in  \dot H^s(\mathbb{R}^n)$, there holds
$$\left\| (-\Delta)^{s/2} f \right\|_2^2-\mathcal S_{s,n} \|f\|_{\frac{2n}{n-2s}}^2\geq \frac{\beta_{s}}{n} \inf_{h\in\mathcal{ M}_s}\|(-\Delta)^{s/2}(f-h)\|_2^2.$$
\end{theorem}

\begin{remark}
 From Chen, Frank and Weth's work in \cite{ChFrWe}, we know
that the optimal lower bound $C_{n,s}$ satisfies
$$C_{n,s}\leq \frac{4s}{n+2+4s}.$$ Recently, K$\ddot{o}$nig in \cite{Ko} proved that the optimal lower bound is strictly smaller than $\frac{4s}{n+2+4s}$. So the lower bound for stability of the fractional Sobolev inequality we obtain in Theorem~\ref{sta of so} is optimal when $n\rightarrow +\infty$. At the same time, the lower bound for stability of the HLS inequality we obtain in Theorem~\ref{sta of hls} is also optimal when $n\rightarrow +\infty$ (see Remark 1.5 in \cite{CLT2}). We also note that K$\ddot{o}$nig \cite{Ko1} furthermore derived the attainability of $C_{n,s}$.
 \end{remark}

 As an application of Theorem \ref{sta of hls}, gathering the results of Theorem 1.4 in \cite{CLT2}, applying the dual stability as we did in \cite{CLT2}, we can further derive the optimal stability of Beckner's \cite{Beckner1} restrictive Sobolev inequality on the flat sub-manifold $\mathbb{R}^{n-1}$ and the sphere $\mathbb{S}^{n-1}$ with dimension-dependent constants.

\begin{corollary}\label{cor}
For any fixed $s\in(0,n/2)$, there exists a positive constant $\beta_s$ such that for any $f\in  \dot H^s(\mathbb{R}^n)\setminus M_{\mathbb{R}^{n-1}}$, there holds
\begin{equation*}\begin{split}
&\int_{\mathbb{R}^n}|(-\Delta)^{\frac{s}{2}}f|^2dx- C_{n,s}^{-1}\big(\int_{\mathbb{R}^{n-1}}|f(x',0)|^{\frac{2(n-1)}{n-2s}}dx'\big)^{\frac{n-2s}{n-1}}\\
&\ \ \geq \frac{\beta_s}{n}\int_{\mathbb{R}^n}|(-\Delta)^{\frac{s}{2}}(f-h)|^2dx,
\end{split}\end{equation*}
where $M_{\mathbb{R}^{n-1}}$ denotes the space consisting of the extremal functions of sharp restrictive Sobolev inequality on flat sub-manifold $\mathbb{R}^{n-1}$.
\end{corollary}

\begin{corollary}\label{cor2}
For any fixed $s\in(0,n/2)$, there exists a positive constant $\beta_s$ such that for an $f\in \dot{H}^s(\mathbb{R}^n)\setminus M_{\mathbb{S}^{n-1}}$, there holds
\begin{equation*}\begin{split}
&\int_{\mathbb{R}^n}|(-\Delta)^{\frac{s}{2}}f|^2dx-D_{n,s}^{-1} \big(\int_{\mathbb{S}^{n-1}}|f(\xi)|^{\frac{2(n-1)}{n-2s}}d\sigma_{\xi}\big)^{\frac{n-2s}{n-1}}\\
&\ \ \geq \frac{\beta_s}{n}\inf_{h\in M_{\mathbb{S}^{n-1}}} \int_{\mathbb{R}^n}|(-\Delta)^{\frac{s}{2}}(f-h)|^2dx,
\end{split}\end{equation*}
where $M_{\mathbb{S}^{n-1}}$ denotes the space consisting of the extremal functions of sharp restrictive Sobolev inequality on the sphere $\mathbb{S}^{n-1}$.
\end{corollary}

Let us give a brief overview over the main ideas of the proof of Theorem \ref{sta of hls}. Its basic strategy is to establish the global stability of the HLS inequality with the optimal lower bounds from the local stability of the HLS inequality with optimal lower bounds. In \cite{CLT2},  we achieve this  for $0<s<1$ by establishing the local stability of the HLS inequality with optimal lower bounds from the local stability of the Sobolev inequality with optimal lower bounds. In the proof of Theorem   \ref{sta of hls}, the main difficulty is to set up the local stability of HLS inequality with the optimal asymptotic lower bounds for $1\leq s<n/2$. While the proof of the local stability of the HLS inequality when $0<s< 1$ in \cite{CLT2} attributes to the local stability of fractional Sobolev inequality by local dual stability and use ``cuttings" at various heights, this method does not  apply to the case for  $1\le s<n/2$. To circumvent this difficulty, we will directly establish the local stability of HLS inequality with optimal asymptotic lower bounds, without using the duality method and the local stability of Sobolev inequality. This is a new and  substantially different strategy we are developing in this paper from those in \cite{CLT1, CLT2}.

\vskip0.1cm

More specifically,  for any $g\in L^{\frac{2n}{n+2s}}(\mathbb{R}^n)$, we decompose $g$ as $g=\phi+r$ with $\phi \in \mathcal{M}_{HLS}$ such that $\|(-\Delta)^{-s/2}(g-\phi)\|_{2}=\inf\limits_{h\in\mathcal{M}_{HLS}}\|(-\Delta)^{-s/2}(g-h)\|_{2}$. We first prove the local stability of the HLS inequality: there exists some $\delta\in(0,1)$ such that if $\|r\|_{\frac{2n}{n+2s}}\leq \delta \|g\|_{\frac{2n}{n+2s}}$, then
$$\|g\|^2_{\frac{2n}{n+2s}}-\mathcal{S}_{s,n}\|(-\Delta)^{-s/2}g\|_2^2\geq \frac{\beta_s}{n}\inf_{h\in M_{HLS}}\|g-h\|^2_{\frac{2n}{n+2s}}.$$
In fact, we prove an equivalent local stability of the  HLS  inequality on the sphere. We must point out that since the decomposition $g=\phi+r$ may not be unique, so $\|r\|_{\frac{2n}{n+2s}}$ is not necessarily unique. However, we can prove the crucial estimates  that all these $\|r\|_{\frac{2n}{n+2s}}$ are comparable (see Lemma \ref{comparable}).

The reason why we decompose $g$ in $H^{-s}(\mathbb{R}^n)$ rather than in $L^{\frac{2n}{n+2s}}(\mathbb{R}^n)$ is that the $H^{-s}$ distance can provide us more orthogonality properties. But this  adds  extra difficulties   to deduce the global stability from the local stability using the rearrangement flow because in our setting $\|r\|_{\frac{2n}{n+2s}}$ is not unique and the ``continuity" of the rearrangement flow is now much  harder to prove.

However, we can still establish the global stability of the HLS inequality by the competing symmetry method of Carlen and Loss, the norm comparison theorem (Lemma \ref{comparable}) and weak continuity for the rearrangement flow of HLS integral (Lemma \ref{continuity}) and a concavity argument.
\vskip0.1cm

This paper is organized as follows. Section 2 is devoted to a local version of stability for HLS inequalities with the optimal asymptotic lower bound for all $1\leq s<\frac{n}{2}$. In Section 3, we will give the global stability of HLS inequalities with the optimal asymptotic lower bound. In Section 4, we will establish the stability of the  Sobolev inequality of all the order $1\le s<\frac{n}{2}$ with the optimal asymptotic lower bound. In Section 5, we add the proof of an auxiliary lemma there.
\medskip

\section{A local version of stability for HLS inequalities with the optimal asymptotic lower bound}
In this section we will set up a local stability of HLS inequalities with the
optimal asymptotic lower bound inspired by \cite{DEFFL}, where Dolbeault, Esteban, Figalli, Frank and Loss
obtained the optimal asymptotic lower bound for the local stability of first-order Sobolev
inequality.  For any $f\in L^{\frac{2n}{n+2s}}(\mathbb{R}^n)$, there exists (see Lemma \ref{decompostition}) $\phi\in \mathcal{M}_{HLS} $ such that
$$\inf_{h\in \mathcal{M}_{HLS}}\|(-\Delta)^{-s/2}(f-h)\|_2=\|(-\Delta)^{-s/2}(f-\phi)\|_2.$$
Set $r=g-\phi$. We point out here in this paper when we say a function $f$ with a decomposition $f=\phi+r$, it always means $\phi\in \mathcal{M}_{HLS}$ satisfying $\inf_{h\in \mathcal{M}_{HLS}}\|(-\Delta)^{-s/2}(f-h)\|_2=\|(-\Delta)^{-s/2}(f-\phi)\|_2.$ It should be noted that the decomposition may not be unique.  The aim of this section is to prove the following local stability.
\begin{lemma}\label{local stability}
For any fixed $s\in (1, \frac{n}{2})$, there exist $\delta\in(0,1)$, $n_0$ and $C_s$ which are dependent on $s$ such that for all $0 \leq f\in L^{\frac{2n}{n+2s}}(\mathbb{R}^n)$ with a decomposition $f=\phi+r$ and
\begin{align}\label{distance}
\|r\|_{L^{\frac{2n}{n+2s}}(\mathbb{R}^n)}^2\leq \delta\|f\|_{L^{\frac{2n}{n+2s}}(\mathbb{R}^n)}^2,
\end{align}
there holds
$$\left\| f \right\|_{L^{\frac{2n}{n+2s}}(\mathbb{R}^n)}^2- \mathcal{S}_{s,n}\|(-\Delta)^{-s/2}f\|_2^2 \geq \frac{1}{n}C_s\inf_{h\in \mathcal{M}_{HLS}}\|f-h\|_{L^{\frac{2n}{n+2s}}(\mathbb{R}^n)}^2$$
when $n\geq n_0$.
\end{lemma}We will prove an equivalent version of local HLS stability inequalities on the sphere using the properties of spherical harmonics and complicated orthogonality technique.

\subsection{The stability of HLS inequalities on the sphere}
By the stereographic projection $\mathcal{S}:\mathbb{R}^n\cup\{\infty\}\rightarrow \mathbb{S}^n$, we know that $\mathbb{R}^n$ (or rather $\mathbb{R}^n\cup\{\infty\}$) and $\mathbb{S}^n$ ($\subset \mathbb{R}^{n+1}$) are conformally equivalent. Thus, there exists an equivalent version of HLS stability inequality on $\mathbb{S}^n$.

In fact, for any $f\in L^{\frac{2n}{n+2s}}(\mathbb{R}^n)$, let $F$ be defined on $\mathbb{S}^n$ by(See Lieb and Loss in \cite{LiebLoss})
$$F(\xi)=\left(\frac{1+|\mathcal{S}^{-1}(\xi)|^2}{2}\right)^{\frac{n+2s}{2}}f(\mathcal{S}^{-1}(\xi)),$$
then
$$\int_{\mathbb{R}^n}|f(x)|^{\frac{2n}{n+2s}}dx=|\mathbb{S}^n|\int_{\mathbb{S}^n}|F(\xi)|^{\frac{2n}{n+2s}}d\sigma_{\xi},$$
where $d\sigma_\xi$ denotes the probability measure on $\mathbb{S}^n$.
Moreover if we denote by $\mathcal{P}_{2s}$  the fractional integral operator on the sphere given by
$$\mathcal{P}_{2s}(g)(\eta)=\frac{|\mathbb{S}^n|\Gamma(\frac{n+2s}{2})}{2^{2s}\pi^{\frac n2}\Gamma(s)}\int_{\mathbb{S}^n}\frac{g(\xi)}{|\xi-\eta|^{n-2s}}d\sigma_{\xi},$$
then
$$\mathcal{S}_{s,n}\|(-\Delta)^{-s/2}f\|_2^2=|\mathbb{S}^n|^{\frac{n+2s}{n}}\langle\mathcal{P}_{2s}(F),F\rangle,$$
where $\langle\cdot,\cdot\rangle$ is the inner product in $L^2(\mathbb{S}^n)$. Thus the sharp HLS inequality on $\mathbb{R}^n$ is equivalent to the following sharp HLS inequality on the sphere
$$\|F\|^2_{L^{\frac{2n}{n+2s}}(\mathbb{S}^n)}\geq \langle\mathcal{P}_{2s}(F),F\rangle,$$
with equality holds if and only if
$$F\in M_{HLS}=\left\{c\left(\frac{\sqrt{1-|\xi|^2}}{1-\xi\cdot w}\right)^{\frac{n+2s}{2}}:\, ~\xi\in \mathbb{R}^{n+1}, c\in \mathbb{R}, |\xi|<1\right\}.$$
It is well known that $F=cJ_{\Phi}^{\frac{n+2s}{2n}}$ for some conformal transformation $\Phi$ on $\mathbb{S}^n$, where $J_{\phi}$ is the Jacobian of the map $\phi$. Then
$$\frac{\|f\|^{2}_{L^{\frac{2n}{n+2s}}(\mathbb{R}^n)}-\mathcal{S}_{s,n}\|(-\Delta)^{-s/2}f\|_2^2}{\inf_{h\in \mathcal{M}_{HLS}}\|g-h\|^2_{L^\frac{2n}{n+2s}(\mathbb{R}^n)}}=\frac{\|F\|^{2}_{L^{\frac{2n}{n+2s}}(\mathbb{S}^n)}-\langle\mathcal{P}_{2s}(F),F\rangle}{\inf_{H\in M_{HLS}}\|F-H\|^2_{L^\frac{2n}{n+2s}(\mathbb{S}^n)}},$$
which implies the equivalence of HLS stability inequalities on $\mathbb{R}^n$ and on $\mathbb{S}^n$.

\subsection{Local stability for nonnegative function}

Instead of proving lemma (\ref{local stability}),
 we first prove the following equivalent lemma on the sphere.
For any $ g\in L^{\frac{2n}{n+2s}}(\mathbb{S}^n)$, by Lemma \ref{decompostition} there exists $\phi\in M_{HLS} $ such that
$$\inf_{h\in M_{HLS}}\langle\mathcal{P}_{2s}(g-h),g-h\rangle=\langle\mathcal{P}_{2s}(g-\phi),g-\phi\rangle.$$ Set $g=\phi+r$, the local stability of HLS inequalities on the sphere states

\begin{lemma}\label{local stability on sphere}
For any fixed $s\in [1, \frac{n}{2})$, there exist $\delta\in(0,1)$, $n_0$ and $C_s$ which are dependent on $s$ such that for all $0 \leq g\in L^{\frac{2n}{n+2s}}(\mathbb{S}^n)$ with a decomposition $g=\phi+r$ and
\begin{align}\label{distance}
\|r\|_{L^{\frac{2n}{n+2s}}(\mathbb{S}^n)}^2\leq \delta\|g\|_{L^{\frac{2n}{n+2s}}(\mathbb{S}^n)}^2,
\end{align}
there holds
$$\left\| g \right\|_{L^{\frac{2n}{n+2s}}(\mathbb{S}^n)}^2- \langle\mathcal{P}_{2s}g,g\rangle\geq \frac{1}{n}C_s\inf_{h\in M_{HLS}}\|g-h\|_{L^{\frac{2n}{n+2s}}(\mathbb{S}^n)}^2$$
when $n\geq n_0$.
\end{lemma}

We need to point out here that the constant $\delta=\frac{1}{2}\frac{\delta_0}{1+\delta_0}$, and $\delta_0$ comes from Lemma~\ref{local sta} below. In order to prove Lemma~\ref{local stability on sphere}, we only need to prove Lemma~\ref{local sta}. By conformal and scaling invariance
of the stability  of HLS inequality, we may assume below $g=1+r\geq 0$, and $r$ is orthogonal with spherical harmonics of degrees $0$ and $1$, which can be deduced from Lemma \ref{decompostition}.
Then (\ref{distance}) implies
 $$\|r\|_{\frac{2n}{n+2s}}^2\leq \delta\|1+r\|_{\frac{2n}{n+2s}}^2\leq \delta\left(1+\|r\|_{\frac{2n}{n+2s}}\right)^2\leq 2\delta\left(1+\|r\|_{\frac{2n}{n+2s}}^2\right).$$
That is $\|r\|_{\frac{2n}{n+2s}}^2\leq \frac{2\delta}{1-2\delta}=\delta_0$. Then Lemma \ref{local stability} can be implied by the following lemma.

\begin{lemma}\label{local sta}
For any fixed $s\in [1, \frac{n}{2})$, there exist $\delta_0\in(0,1)$, $n_0$ and $C_s$ which are dependent on $s$ such that for all $-1\leq r\in L^{\frac{2n}{n+2s}}(\mathbb{S}^n)$ satisfying
$$\left(\int_{\mathbb{S}^n} |r|^{\frac{2n}{n+2s}}d\sigma_\xi\right)^{\frac{n+2s}{n}}\leq \delta_0~~~\textit{and}~~~\int_{\mathbb{S}^n}r d\sigma_\xi=0=\int_{\mathbb{S}^n}w_jrd\sigma_\xi, ~~j=1,\cdots,n+1.$$
There holds
\begin{align*}
&\big(\int_{\mathbb{S}^n}(1+r)^{\frac{2n}{n+2s}}d\sigma_\xi\big)^{\frac{n+2s}{n}}-\langle\mathcal{P}_{2s}(1+r),1+r\rangle\geq \frac{1}{n}C_s\big(\int_{\mathbb{S}^n}|r|^{\frac{2n}{n+2s}}d\sigma_\xi\big)^{\frac{n+2s}{n}}
\end{align*}
when $n\geq n_0$.
\end{lemma}
\begin{remark}
As we explained in the introduction, we decompose $g$ in $H^{-s}(\mathbb{S}^n)$ instead of $L^{\frac{2n}{n+2s}}(\mathbb{S}^n)$ because we can obtain simple orthogonality condition.
\end{remark}

Now, let us prove Lemma \ref{local sta}. For any $r\geq -1$, define $r_1$, $r_2$ and $r_3$ by
$$r_1=\min\{r, \gamma\},\ \ r_2=\min\{(r-\gamma)_{+}, M-\gamma\},\ \ r_3=(r-M)_{+},$$
where $\gamma$ and $M$ are two parameters such that $0<\gamma<M$. We need the following lemma.
\begin{lemma}\label{q-estimate}
Given $\varepsilon>0$, $M>0$, $2>p_0>1$, and $\gamma\in (0, \frac{M}{2})$. There exists a positive constant $C_{p_0,\gamma,\varepsilon,M}$ such that for any
$r\geq -1$ and $1<p_0\leq p\leq 2$, there holds
\begin{equation}\begin{split}
(1+r)^p-1-pr&\geq (\frac{1}{2}p(p-1)-2\gamma\theta)r_1^2+(\frac{1}{2}p(p-1)-C_{p_0,\gamma,\varepsilon,M}\theta)r_2^2\\
&\ \ +2r_1r_2+2(r_1+r_2)r_2+(1-\varepsilon\theta)r_3^p,
\end{split}\end{equation}
where $\theta=2-p$.
\end{lemma}
The proof of the lemma will be stated in the last section. The proof is similar to that of Corollary 2.12 in (\cite{DEFFL}), where a reversed estimate for $(1+r)^p-1-pr$ when $p\in [2,3]$ was established.

\subsection{Dividing the deficit}

First we split the double integral part of the deficit
$$\langle\mathcal{P}_{2s}r,r\rangle=\frac{|\mathbb{S}^n|\Gamma(\frac{n+2s}{2})}{2^{2s}\pi^{\frac n2}\Gamma(s)}\int_{\mathbb{S}^n}\int_{\mathbb{S}^n}\frac{r(\xi)r(\eta)}{|\xi-\eta|^{n-2s}}d\sigma_\xi d\sigma_\eta$$
 in the following lemma, which is crucial to our proof.
\begin{lemma}\label{lem split}
There holds
\begin{equation*}\begin{split}
&\langle\mathcal{P}_{2s}(r),r\rangle
\leq \sum_{i=1}^{3}\langle\mathcal{P}_{2s}(r_i),r_i\rangle+2\int_{\mathbb{S}^n}r_1(\xi)r_2(\xi)d\sigma_\xi\\
&\ \ +2\int_{\mathbb{S}^n}r_2(\xi)r_3(\xi)d\xi+2\int_{\mathbb{S}^n}r_1(\xi)r_3(\xi)d\sigma_\xi.
\end{split}\end{equation*}
\end{lemma}

\begin{proof}
First we have
\begin{equation*}
\langle\mathcal{P}_{2s}(r),r\rangle=\sum_{i=1}^{3}\langle\mathcal{P}_{2s}(r_i),r_i\rangle
+2\langle\mathcal{P}_{2s}(r_1),r_2\rangle
 +2\langle\mathcal{P}_{2s}(r_2),r_3\rangle+2\langle\mathcal{P}_{2s}(r_1),r_3\rangle.
\end{equation*}
For the term $\langle\mathcal{P}_{2s}(r_1),r_2\rangle$, using the fact $r_1\leq \gamma$, $r_1(\xi)=\gamma$ when $r(\xi)\geq \gamma$ and $r_2=0$ when $r(\xi)<\gamma$, we have
\begin{equation*}\begin{split}
&\langle\mathcal{P}_{2s}(r_1),r_2\rangle
=\frac{|\mathbb{S}^n|\Gamma(\frac{n+2s}{2})}{2^{2s}\pi^{\frac n2}\Gamma(s)}\int_{\mathbb{S}^n}\int_{\mathbb{S}^n}\frac{r_{1}(\eta)}{|\xi-\eta|^{n-2s}}d\sigma_{\eta}r_{2}(\xi)d\sigma_{\xi}\\
&\ \ =\frac{|\mathbb{S}^n|\Gamma(\frac{n+2s}{2})}{2^{2s}\pi^{\frac n2}\Gamma(s)}\int_{\{r(\xi)\geq \gamma\}}\int_{\mathbb{S}^n}\frac{r_{1}(\eta)}{|\xi-\eta|^{n-2s}}d\sigma_{\eta}r_2(\xi)d\sigma_{\xi}\\
&\ \ \leq \frac{|\mathbb{S}^n|\Gamma(\frac{n+2s}{2})}{2^{2s}\pi^{\frac n2}\Gamma(s)}\int_{\{r(\xi)\geq \gamma\}}\int_{\mathbb{S}^n}\frac{\gamma}{|\xi-\eta|^{n-2s}}d\sigma_{\eta}r_2(\xi)d\sigma_{\xi}\\
&\ \ =\int_{r(\xi)\geq \gamma}\gamma r_2(\xi)d\sigma_{\xi}=\int_{\mathbb{S}^n} r_1(\xi)r_2(\xi)d\sigma_{\xi},
\end{split}\end{equation*}
where we use the fact $\int_{\mathbb{S}^n}|\xi-\eta|^{-n+2s}d\sigma_{\eta}=\frac{2^2s\pi^{n/2}\Gamma(s)}{\Gamma(n/2+s)|\mathbb{S}^n|}$.
Similarly, we can also derive that
\begin{equation*}\begin{split}
&\langle\mathcal{P}_{2s}(r_1),r_3\rangle
 =\frac{|\mathbb{S}^n|\Gamma(\frac{n+2s}{2})}{2^{2s}\pi^{\frac n2}\Gamma(s)}\int_{\mathbb{S}^n}\int_{\mathbb{S}^n}\frac{r_{1}(\eta)}{|\xi-\eta|^{n-2s}}d\sigma_{\eta}r_{3}(\xi)d\sigma_{\xi}\\
&\ \ \leq \frac{|\mathbb{S}^n|\Gamma(\frac{n+2s}{2})}{2^{2s}\pi^{\frac n2}\Gamma(s)}\int_{\{r(\xi)\geq M\}}\int_{\mathbb{S}^n}\frac{\gamma}{|\xi-\eta|^{n-2s}}d\sigma_{\eta}r_3(\xi)d\sigma_{\xi}\\
&\ \ =\int_{\{r(\xi)\geq M\}}\gamma r_3(\xi)d\sigma_{\xi}=\int_{\mathbb{S}^n} r_1(\xi)r_3(\xi)d\sigma_{\xi},
\end{split}\end{equation*}

and \begin{equation*}\begin{split}
&\langle\mathcal{P}_{2s}(r_2),r_3\rangle
 =\frac{|\mathbb{S}^n|\Gamma(\frac{n+2s}{2})}{2^{2s}\pi^{\frac n2}\Gamma(s)}\int_{\{r(\xi)\geq M\}}\int_{\mathbb{S}^n}\frac{r_{2}(\eta)}{|\xi-\eta|^{n-2s}}d\sigma_{\eta}r_3(\xi)d\sigma_{\xi}\\
&\ \ \leq \frac{|\mathbb{S}^n|\Gamma(\frac{n+2s}{2})}{2^{2s}\pi^{\frac n2}\Gamma(s)}\int_{\{r(\xi)\geq M\}}\int_{\mathbb{S}^n}\frac{M-\gamma}{|\xi-\eta|^{n-2s}}d\sigma_{\eta}r_3(\xi)d\sigma_{\xi}\\
&\ \ =\int_{\{r(\xi)\geq M\}}(M-\gamma) \big(r(\xi)-M\big)d\sigma_{\xi}=\int_{\mathbb{S}^n} r_2(\xi)r_3(\xi)d\sigma_{\xi}.
\end{split}\end{equation*}
Then we accomplish the proof of Lemma \ref{lem split}.
\end{proof}

Then let us handle the  terms $\langle\mathcal{P}_{2s}(1+r),1+r\rangle$ and $\left(\int_{\mathbb{S}^{n}}(1+r)^{p}\,d\sigma_{\xi}\right)^{2/p}$. Since $r$ has mean zero, then
$\langle\mathcal{P}_{2s}(r),1\rangle=0$. Along with Lemma (\ref{lem split}),
direct computation gives that
\begin{align}\label{est of L2}\nonumber
& \langle\mathcal{P}_{2s}(1+r),1+r\rangle=\langle\mathcal{P}_{2s}(1),1\rangle + \langle\mathcal{P}_{2s}(r),r\rangle\\\nonumber
&\leq 1+\sum_{i=1}^{3}\langle\mathcal{P}_{2s}(r_i),r_i\rangle+2\int_{\mathbb{S}^n}r_1(\xi)r_2(\xi)d\sigma_\xi\\
&\ \ +2\int_{\mathbb{S}^n}r_2(\xi)r_3(\xi)d\xi+2\int_{\mathbb{S}^n}r_1(\xi)r_3(\xi)d\sigma_\xi.
\end{align}

Denote $p=\frac{2n}{n+2s}$, $\theta=2-p=\frac{4s}{n+2s}$ and assume $n>14s$, then $\frac 74<p<2$ and $0<\theta<1$. Given two parameters $\epsilon_1>0$, $\epsilon_2>0$, we apply Lemma \ref{q-estimate} with (an arbitrary choice of $M\geq 2\gamma$),
$$\gamma=\frac{\epsilon_1}{2},~~~~\epsilon=\epsilon_2,~~~C_{p_0,\gamma,\epsilon,M}=C_{\epsilon_1,\epsilon_2} ,$$
the inequality $(1+x)^{\frac{2}{p}}\geq 1+\frac{2}{p}x$ and $\int_{\mathbb{S}^n}r d\sigma_\xi=0$
to derive that
\begin{align}\label{est of Lq}\nonumber
&\left(\int_{\mathbb{S}^{n}}(1+r)^{p}\,d\sigma_{\xi}\right)^{2/p}\\\nonumber
 &\ \ \geq 1 + (p-1-\frac{2}{p}\epsilon_1\theta)\int_{\mathbb{S}^{n}}r_1^{2}\,d\sigma_{\xi} + (p-1-\frac{2}{p}C_{\epsilon_1,\epsilon_2}\theta)\int_{\mathbb{S}^{n}}r_2^{2}\,d\sigma_{\xi}  \\\nonumber
 &\ \ \ \ + \frac{4}{p}\int_{\mathbb{S}^{n}}(r_1r_2)\,d\sigma_{\xi} + \frac{4}{p}\int_{\mathbb{S}^{n}}(r_1+r_2)r_3\,d\sigma_{\xi} + \frac{2}{p}(1-\epsilon_2\theta)\int_{\mathbb{S}^{n}}r_3^{p}\,d\sigma_{\xi}
\\&\ \ \geq 1 + (p-1-\frac{2}{p}\epsilon_1\theta)\int_{\mathbb{S}^{n}}r_1^{2}\,d\sigma_{\xi} + (p-1-\frac{2}{p}C_{\epsilon_1,\epsilon_2}\theta)\int_{\mathbb{S}^{n}}r_2^{2}\,d\sigma_{\xi}  \\\nonumber
&\ \ \ \ \ + 2\int_{\mathbb{S}^{n}}r_1r_2\,d\sigma_{\xi} + 2\int_{\mathbb{S}^{n}}(r_1+r_2)r_3\,d\sigma_{\xi} + \frac{2}{p}(1-\epsilon_2\theta)\int_{\mathbb{S}^{n}}r_3^{p}\,d\sigma_{\xi}.
\end{align}
Combining Lemma \ref{lem split}, (\ref{est of L2}) and (\ref{est of Lq}), we can divide the deficit as follows,
\begin{footnotesize}
\begin{align*}
&\left(\int_{\mathbb{S}^{n}}(1+r)^{p}\,d\sigma_\xi\right)^{2/p}-\langle\mathcal{P}_{2s}(1+r),1+r\rangle\\
&\ \ \geq \left(\frac{2}{p}\epsilon_1\vartheta\theta \int_{\mathbb{S}^n}|r_1|^2d\sigma_{\xi}+\frac{2}{p}C_{\epsilon_1,\epsilon_2}\vartheta\theta \int_{\mathbb{S}^n}|r_2|^2d\sigma_{\xi}+\frac{2}{p}\epsilon_2\vartheta\theta \int_{\mathbb{S}^n}|r_3|^pd\sigma_{\xi}\right)\\
 &\ \ \ \ +\left(p-1-\frac{2}{p}\epsilon_1(1+\vartheta)\theta\right)\int_{\mathbb{S}^{n}}r_1^{2}\,d\sigma_{\xi}-\langle\mathcal{P}_{2s}(r_1),r_1\rangle\\
 &\ \ \ \  +\left(p-1-\frac{2}{p}C_{\epsilon_1,\epsilon_2}(1+\vartheta)\theta\right)\int_{\mathbb{S}^{n}}r_2^{2}\,d\sigma_{\xi}-\langle\mathcal{P}_{2s}(r_2),r_2\rangle\\
&\ \ \ \ +\frac{2}{p}\left(1-\epsilon_2(1+\vartheta)\theta\right)\int_{\mathbb{S}^{n}}r_3^{p}\,d\sigma_{\xi}-\langle\mathcal{P}_{2s}(r_3),r_3\rangle.\\
\end{align*}
\end{footnotesize}
Let us define
\begin{align*}
I_1:&=\left(p-1-\frac{2}{p}\epsilon_1(1+\vartheta)\theta\right)\int_{\mathbb{S}^{n}}r_1^{2}\,d\sigma_{\xi}-\langle\mathcal{P}_{2s}(r_1),r_1\rangle\\
 &\ \ +\sigma_0\theta\int_{\mathbb{S}^{n}}(r_2^{2}+r_3^p)d\sigma_{\xi},
\\I_2&:=\left(p-1-\frac{2}{p}C_{\epsilon_1,\epsilon_2}(1+\vartheta)\theta-\sigma_0\theta\right)\int_{\mathbb{S}^{n}}r_2^{2}\,d\sigma_{\xi}-\langle\mathcal{P}_{2s}(r_2),r_2\rangle,\\
\\I_3&:=\left(\frac{2}{p}\left(1-\epsilon_2(1+\vartheta)\theta\right)-\sigma_0\theta\right)\int_{\mathbb{S}^{n}}r_3^{p}\,d\sigma_{\xi}-\langle\mathcal{P}_{2s}(r_3),r_3\rangle,\\
\end{align*}
where the parameter $\sigma_0 > 0$ will be determined later. To summarize, we have
\begin{align*}
&\left(\int_{\mathbb{S}^{n}}(1+r)^{p}\,d\sigma_\xi\right)^{2/p}-\langle\mathcal{P}_{2s}(1+r),1+r\rangle\\
&\ \ \geq \left(\frac{2}{p}\epsilon_1\vartheta\theta \int_{\mathbb{S}^n}|r_1|^2d\sigma_{\xi}+\frac{2}{p}C_{\epsilon_1,\epsilon_2}\vartheta\theta \int_{\mathbb{S}^n}|r_2|^2d\sigma_{\xi}+\frac{2}{p}\epsilon_2\vartheta\theta \int_{\mathbb{S}^n}|r_3|^pd\sigma_{\xi}\right)+ \sum\limits_{k=1}^{3}I_k.
\end{align*}
Using the fact $1>\|r_3\|_{p}^{p}\geq \|r_3\|_p^2$ and $\|r_1\|_2\geq \|r_1\|_p,~~\|r_2\|_2\geq \|r_2\|_p$ by H$\ddot{o}$lder's inequality, we have
$$\left(\frac{2}{p}\epsilon_1\vartheta\theta \int_{\mathbb{S}^n}|r_1|^2d\sigma_{\xi}+\frac{2}{p}C_{\epsilon_1,\epsilon_2}\vartheta\theta \int_{\mathbb{S}^n}|r_2|^2d\sigma_{\xi}+\frac{2}{p}\epsilon_2\vartheta\theta \int_{\mathbb{S}^n}|r_3|^pd\sigma_{\xi}\right)\geq C_{\epsilon_1,\epsilon_2,\vartheta}\theta\|r\|_{p}^2.$$
In the following, we only need to show that $I_3$, $I_1$ and $I_2$ are nonnegative respectively. More precisely, we will prove that for $\epsilon_1\leq \frac{1}{16}$, $\epsilon_2\leq \frac{1}{8}$, $\vartheta\leq \frac{1}{2}$ and $n\geq 14 s$,
there exists a constant $0<\delta_0<1$, independent of $n$, such that for  all $-1\leq r\in L^{\frac{2n}{n+2s}}(\mathbb{S}^n)$ satisfying
$$\left(\int_{\mathbb{S}^n}|r|^{\frac{2n}{n+2s}}d\sigma_\xi\right)^{\frac{n+2s}{n}}\leq \delta_0~~~\textit{and}~~~\int_{\mathbb{S}^n}r d\sigma_\xi=0=\int_{\mathbb{S}^n}w_jrd\sigma_\xi, ~~j=1,\cdots,n+1,$$
there holds $I_i\geq 0$, $i=1,2,3$.

\subsection{Bound on $I_3$}
Let us prove $I_3\geq 0$. Choose $\vartheta=\frac{1}{2}$, $\epsilon_2\leq \frac{1}{8}$ and $\sigma_0=\frac{1}{8}$. Then
\begin{equation}\begin{split}\nonumber
\left(\frac{2}{p}\left(1-\epsilon_2(1+\vartheta)\theta\right)-\sigma_0\theta\right)&=1+\frac{2s}{n}-\frac{n+2s}{n}\epsilon_2(1+\vartheta)\frac{4s}{n+2s}-\sigma_0 \frac{4s}{n+2s}\\
&\geq 1+\frac{1}{4}\frac{2s}{n}.
\end{split}\end{equation}
Using the fact $\|r_3\|^p_{L^p(\mathbb{S}^n)}\geq \|r_3\|^2_{L^p(\mathbb{S}^n)}$ and the Hardy-Littlewood-Sobolev inequality, we obtain
\begin{equation*}\begin{split}
I_3&=\left(\frac{2}{p}\left(1-\epsilon_2(1+\vartheta)\theta\right)-\sigma_0\theta\right)\int_{\mathbb{S}^{n}}r_3^{p}\,d\sigma_{\xi}-\langle\mathcal{P}_{2s}(r_3),r_3\rangle\\
&\geq \left(\frac{2}{p}\left(1-\epsilon_2(1+\vartheta)\theta\right)-\sigma_0\theta\right)\|r_3\|_{L^p(\mathbb{S}^n)}^2-\langle\mathcal{P}_{2s}(r_3),r_3\rangle\\
&\geq \|r_3\|_{L^p(\mathbb{S}^n)}^2-\langle\mathcal{P}_{2s}(r_3),r_3\rangle\geq 0.
\end{split}\end{equation*}

\subsection{Bound on $I_1$}
In this subsection, we will prove $I_1\geq 0$. Let $\tilde{Y}_{l,m}$ be the $L^{2}$-normalized spherical harmonics with respect to the uniform probability measure on the sphere for any $m=1,2,...,N(n,l)$, where
$$N(n,0)=1~~\text{and}~~N(n,l)=\frac{(2l+n-1)\Gamma(l+n-1)}{\Gamma(l+1)\Gamma(n)}.$$
 Let $\widetilde{r_1}$ be the orthogonal projection of $r_1$ onto the space of spherical harmonics of degree $\geq 2$, that is,
\begin{align*}
\widetilde{r_1}=r_1 - \int_{\mathbb{S}^{n}}r_1d\sigma_\xi - \sum_{j=1}^{n+1}\big(\tilde{Y}_{1,j}\int_{\mathbb{S}^{n}}r_1\tilde{Y}_{1,j}d\sigma_\xi\big).
\end{align*}
By the orthogonality of spherical harmonics, we have
\begin{align}\label{sph of r_1}
\int_{\mathbb{S}^n}r_1^2d\sigma_\xi=\int_{\mathbb{S}^n}\widetilde{r_1}^2d\sigma_\xi+(\int_{\mathbb{S}^n}r_1d\sigma_\xi)^2+\sum_{j=1}^{n+1}\big(\int_{\mathbb{S}^{n}}r_1\tilde{Y}_{1,m}d\sigma_\xi\big)^2.
\end{align}
By the Funk-Hecke formula (see \cite[Eq.~(17)]{Be1993} and also  \cite[Corollary 4.3]{FL2}), we have
\begin{align}\label{Funk-Heck}
\langle\mathcal{P}_{2s}(r_1),r_1\rangle=\sum_{l=1}^{\infty}A_{n,s}(l)\sum_{m=1}^{N(n,l)}(\int_{\mathbb{S}^n}r_1\tilde{Y}_{l,m}d\sigma_\xi)^2
\end{align}
where $A_{n,s}(l)=\frac{\Gamma(\frac{n}{2}+s)\Gamma(\frac{n}{2}-s+l)}{\Gamma(\frac{n}{2}-s)\Gamma(\frac{n}{2}+s+l)}$.

Then using (\ref{sph of r_1}), (\ref{Funk-Heck}) and the fact that $A_{n,s}(l)$ is decreasing on $l$, we can estimate $I_1$ as follows,

\begin{align*}
I_1&=\left(p-1-\frac{2}{p}\epsilon_1(1+\vartheta)\theta\right)\int_{\mathbb{S}^{n}}r_1^{2}\,d\sigma_{\xi}-\langle\mathcal{P}_{2s}(r_1),r_1\rangle
+\sigma_0\theta\int_{\mathbb{S}^{n}}(r_2^{2}+r_3^p)d\sigma_{\xi}\\
&\geq \left(p-1-\frac{2}{p}\epsilon_1(1+\vartheta)\theta\right)\int_{\mathbb{S}^n}\widetilde{r_1}^2d\sigma_\xi+\left(p-1-\frac{2}{p}\epsilon_1(1+\vartheta)\theta\right)\sum_{j=1}^{n+1}\big(\int_{\mathbb{S}^n}r_1\tilde{Y}_{1,j}d\sigma_\xi\big)^2\\
&\ \ +\left(p-1-\frac{2}{p}\epsilon_1(1+\vartheta)\theta\right)(\int_{\mathbb{S}^n}r_1d\sigma_\xi)^2 -A_{n,s}(2)\int_{\mathbb{S}^n}\widetilde{r_1}^2d\sigma_\xi-A_{n,s}(0)(\int_{\mathbb{S}^n}r_1d\sigma_\xi)^2\\
&\ \ -A_{n,s}(1)\sum_{j=1}^{n+1}\big(\int_{\mathbb{S}^{n}}r_1\tilde{Y}_{1,j}d\sigma_\xi\big)^2+\frac{4s}{n+2s}\sigma_0\int_{\mathbb{S}^{n}}(r_2^{2}+r_3^{p})d\sigma_\xi,
\end{align*}
where $$A_{n,s}(1)=\frac{\Gamma(\frac{n}{2}+s)\Gamma(\frac{n}{2}-s+1)}{\Gamma(\frac{n}{2}-s)\Gamma(\frac{n}{2}+s+1)}=1-\frac{4s}{n+2s}$$
and $$A_{n,s}(2)=\frac{\Gamma(\frac{n}{2}+s)\Gamma(\frac{n}{2}-s+2)}{\Gamma(\frac{n}{2}-s)\Gamma(\frac{n}{2}+s+2)}=(1-\frac{4s}{n+2s})\cdot(1-\frac{4s}{n+2s+2}).$$
Since
\begin{equation}\begin{split}\nonumber
& p-1-\frac{2}{p}\epsilon_1(1+\vartheta)\theta-A_{n,s}(2)\\
&= \frac{4s}{n+2+2s}-\frac{n+2s}{n}\epsilon_1(1+\vartheta)\frac{4s}{n+2s}-\frac{16s^2}{(n+2+2s)(n+2s)},\\
&\geq \frac{4s}{n+2s}(1/4-2\epsilon_1),
\end{split}\end{equation}
$$\left(\frac{4s}{n+2s}\right)^{-1}\left(p-1-\frac{2}{p}\epsilon_1(1+\vartheta)\theta-A_{n,s}(1)\right)=-\frac{n+2s}{n}\epsilon_1(1+v),$$
and $$\left(\frac{4s}{n+2s}\right)^{-1}\left(p-2-\frac{2}{p}\epsilon_1(1+\vartheta)\theta\right)=-1-\frac{n+2s}{n}\epsilon_1(1+v),$$
recalling that $\vartheta=1/2$, $\epsilon_1\leq 1/16$ and $n\geq 14s$,  we have
\small{\begin{equation}\label{I_1}\begin{split}
&I_1\left(\frac{4s}{n+2s}\right)^{-1}\\
&\geq \left(\frac{4s}{n+2s}\right)^{-1}\left(p-1-\frac{2}{p}\epsilon_1(1+\vartheta)\theta-A_{n,s}(1)\right)\sum_{j=1}^{n+1}\big(\int_{\mathbb{S}^n}r_1\tilde{Y}_{1,j}d\sigma_\xi\big)^2\\
&\ \ \ \ \ +\left(\frac{4s}{n+2s}\right)^{-1}\left(p-2-\frac{2}{p}\epsilon_1(1+\vartheta)\theta\right)(\int_{\mathbb{S}^n}r_1d\sigma_\xi)^2+\sigma_0\int_{\mathbb{S}^{n}}(r_2^{2}+r_3^{p})d\sigma_\xi\\
&\ \ \geq \frac{1}{8}\int_{\mathbb{S}^n}\widetilde{r_1}^2d\sigma_\xi
-\frac{17}{14}(\int_{\mathbb{S}^n}r_1d\sigma_\xi)^2-\frac{3}{14}\sum_{j=1}^{n+1}\big(\int_{\mathbb{S}^{n}}r_1\tilde{Y}_{1,j}d\sigma_\xi\big)^2\\
&\ \ \ \ \ +\sigma_0\int_{\mathbb{S}^{n}}(r_2^{2}+r_3^{p})d\sigma_\xi.\\
\end{split}\end{equation}}
 Next we will show that $I_1 \geq 0$.
Let $Y$ be one of the functions 1 and $\sum_{j=1}^{n+1}a_j\tilde{Y}_{1,j}$, where $a_j \in \mathbb{R}$. Using
the fact $$\int_{\mathbb{S}^{n}}Yrd\sigma_\xi=0$$
and H$\ddot{o}$lder's inequality we can get
\begin{equation}\begin{split}\label{adint1}
\left(\int_{\mathbb{S}^{n}}Yr_1\,d\sigma_\xi\right)^{2}& = \left(\int_{\mathbb{S}^{n}}Y(r_2+r_3)\,d\sigma_\xi \right)^{2}\\
 &\leq \|Y\|^{2}_{L^{\frac{2p}{2p-3}}(\mathbb{S}^{n})}\left(\int_{\{r_2+r_3>0\}}d\sigma_\xi\right)^{\frac{1}{p}} \|r_2+r_3\|^{2}_{L^{p}(\mathbb{S}^{n})}\\
 & \leq \|Y\|^{2}_{L^{7}(\mathbb{S}^{n})}\left(\int_{\{r_2+r_3>0\}}d\sigma_\xi\right)^{\frac{1}{p}} \|r_2+r_3\|^{2}_{L^{p}(\mathbb{S}^{n})},
\end{split}\end{equation}
where in the last inequality we have used the fact $\frac{2p}{2p-3}\leq 7$ when $n>14s$. Since $\{r_2+r_3>0\} \subset \{r_1 \geq \gamma\}$, we have
\begin{equation}\label{adint2}\begin{split}
\int_{\{r_2+r_3>0\}}d\sigma_\xi \leq \int_{\{r_1\geq \gamma\}}d\sigma_\xi \leq \frac{1}{\gamma^{p}}\int_{\mathbb{S}^{n}}|r_1|^{p}\,d\sigma_\xi = \frac{1}{\gamma^{p}}\|r_1\|^{p}_{L^{p}(\mathbb{S}^{n})}.
\end{split}\end{equation}
By \eqref{adint1} and \eqref{adint2},  we derive that
\begin{align*}
\left(\int_{\mathbb{S}^{n}}Yr_1\,d\sigma_\xi\right)^{2} \leq \|Y\|^{2}_{L^{7}(\mathbb{S}^{n})}\frac{1}{\gamma}\|r_1\|_{L^{p}(\mathbb{S}^{n})} \|r_2+r_3\|^{2}_{L^{p}(\mathbb{S}^{n})}.
\end{align*}
Note
\begin{equation}\begin{split}\nonumber
\int_{\mathbb{S}^n}|r_2+r_3|^pd\sigma_{\xi}&=\int_{\{r(\xi)\geq \gamma\}}|r_2+r_3|^pd\sigma_{\xi}\\
&\leq \int_{\{r(\xi)\geq \gamma\}}|r_1+r_2+r_3|^pd\sigma_{\xi}\\
&\leq \int_{\mathbb{S}^n}|r|^pd\sigma_{\xi}\leq \delta_0^{\frac{p}{2}}\leq \delta_0^{\frac{1}{2}}.
\end{split}\end{equation}
Combining the above estimates, we get that
\begin{equation}\label{int1}
\left(\int_{\mathbb{S}^{n}}Yr_1\,d\sigma_\xi\right)^{2}\leq \|Y\|^{2}_{L^{7}(\mathbb{S}^{n})}\frac{\sqrt{\delta_0}}{\gamma} \|r_1\|_{L^{p}(\mathbb{S}^{n})}\|r_2+r_3\|_{L^p(\mathbb{S}^n)}.
\end{equation}
If $Y = 1$, \eqref{int1} directly gives that
\begin{align}\label{adint3}
\left(\int_{\mathbb{S}^{n}}r_1\,d\sigma_\xi\right)^{2} \leq \frac{\sqrt{\delta_0}}{\gamma} \|r_1\|_{L^{p}(\mathbb{S}^{n})}\|r_2+r_3\|_{L^p(\mathbb{S}^n)}.
\end{align}
From [\cite{Du}, Theorem 1], we know that for any $L^2$-normalized spherical harmonic $Y$ of degree $k\in N$, we have $\|Y\|_{L^p(\mathbb{S}^n)}\leq (p-1)^{\frac{k}{2}}$ for any $p\geq 2$. Then we can deduce that if $Y = \sum_{j=1}^{n+1}a_j\tilde{Y}_{1,j}$, then
$$\|Y\|_{L^7(\mathbb{S}^n)}=(\sum_{j=1}^{n+1}a_j^2)^{\frac{1}{2}}\|\frac{Y}{\|Y\|_{L^2(\mathbb{S}^n)}}\|_{L^7(\mathbb{S}^n)}\leq (\sum_{j=1}^{n+1}a_j^2)^{\frac{1}{2}}\sqrt{6}.$$
Pick $a_j=\int_{\mathbb{S}^{n}}  \tilde{Y}_{1,j}r_1d\sigma_\xi$, then it follows that
\begin{equation}\label{adint4}\begin{split}
\sum_{j=1}^{n+1}\big(\int_{\mathbb{S}^{n}}r_1\tilde{Y}_{1,j}d\sigma_\xi\big)^2&= (\sum_{j=1}^{n+1}a_j^2)^{-1}(\int_{\mathbb{S}^{n}}Yr_1\,d\sigma_\xi)^{2}\\
 &\leq (\sum_{j=1}^{n+1}a_j^2)^{-1} \frac{\sqrt{\delta_0}}{\gamma}\|Y\|_{L^7(\mathbb{S}^n)}^2 \|r_1\|_{L^{p}(\mathbb{S}^{n})}\|r_2+r_3\|_{L^p(\mathbb{S}^n)}\\
 &\leq \frac{6\sqrt{\delta_0}}{\gamma}\|r_1\|_{L^{p}(\mathbb{S}^{n})}\|r_2+r_3\|_{L^p(\mathbb{S}^n)}.
\end{split}\end{equation}
Gathering \eqref{I_1}, \eqref{adint3} and \eqref{adint4} and the fact $\|r_2\|_{L^2(\mathbb{S}^n)}^2\geq\|r_2\|_{L^p(\mathbb{S}^n)}^2$, $\|r_3\|_{L^p(\mathbb{S}^n)}^p\geq\|r_3\|_{L^p(\mathbb{S}^n)}^2$, we conclude that
\small{\begin{equation}\begin{split} \nonumber
I_1\big(\frac{4s}{n+2s}\big)^{-1}
&\geq \frac{1}{8}\int_{\mathbb{S}^n}r_1^2d\sigma_\xi+\sigma_0\left(\|r_2\|_{L^p(\mathbb{S}^n)}^2+\|r_3\|_{L^p(\mathbb{S}^n)}^2\right)\\
&\ \ -\frac{17}{14}\frac{\sqrt{\delta_0}}{\gamma} \|r_1\|_{L^{p}(\mathbb{S}^{n})}\|r_2+r_3\|_{L^p(\mathbb{S}^n)}\\
&\ \ -\frac{3}{14}\frac{6\sqrt{\delta_0}}{\gamma}\|r_1\|_{L^{p}(\mathbb{S}^{n})}\|r_2+r_3\|_{L^p(\mathbb{S}^n)}.
\end{split}\end{equation}}
Then by the Cauchy-Swartz inequality, $I_1$ is nonnegative if
\begin{align}\label{ine for I_1}
2\sqrt{\frac{1}{8}}\sqrt{\sigma_0}\frac{1}{\sqrt{2}}\geq \frac{17}{14}\frac{\sqrt{\delta_0}}{\gamma}+\frac{3}{14}\frac{6\sqrt{\delta_0}}{\gamma}.
\end{align}
Recall that $\sigma_0=\frac{1}{8}$, $\gamma=\frac {\epsilon_1}{2}$ and denote
\begin{align}\label{delta1}
\delta_1:= \frac{98\epsilon_1}{(280)^2}.
\end{align}
Let us choose $\delta_0\leq \delta_1$, then
 inequality (\ref{ine for I_1}) holds. It is obvious that $\delta_1$ is independent of $n$.

\subsection{Bound on $I_2$}
In this subsection, we will prove $I_2\geq 0$. From \cite{DEFFL} ( see inequality (25) in \cite{DEFFL}), we know that for any $L^{2}$-normalized spherical harmonic $Y$ of degree $k \in \mathbb{N}$
\begin{align}\nonumber
\int_{\mathbb{S}^{n}}Yr_2\,d\sigma_{\xi} \leq 3^{\frac{k}{2}}\gamma^{-\frac{p}{4}}\delta_0^{\frac{p}{8}}\|r_2\|_{L^{2}(\mathbb{S}^{n})}.
\end{align}
Let us denote by $\pi_{k}r_2$ the projection of $r_2$ onto spherical harmonics of degree $k$, and let $Y=\frac{\pi_{k}r_2}{\|\pi_{k}r_2\|_2}$, then there holds
\begin{align*}
\|\pi_kr_2\|_{L^{2}(\mathbb{S}^{n})} \leq 3^{\frac{k}{2}}\gamma^{-\frac{p}{4}}\delta_0^{\frac{p}{8}}\|r_2\|_{L^{2}(\mathbb{S}^{n})}.
\end{align*}
For any $K \in \mathbb{N}$, we denote by $\Pi_Kr_2:= \sum_{k<K}\pi_{k}r_2$ the projection of $r_2$ onto spherical harmonics of degree less than $K$, then we can derive that
\begin{equation}\begin{split}\label{int2}
\|\Pi_Kr_2\|_{L^{2}(\mathbb{S}^{n})}&= (\sum_{k<K}\|\pi_kr_2\|^{2}_{L^{2}(\mathbb{S}^{n})})^{1/2}\\
 &\leq \gamma^{-\frac{p}{4}}\delta_0^{\frac{p}{8}}\|r_2\|_{L^{2}(\mathbb{S}^{n})}\sqrt{\sum_{k<K}3^{k}}\\
  &\leq 3^{\frac{K}{2}}\gamma^{-\frac{p}{4}}\delta_0^{\frac{p}{8}}\|r_2\|_{L^{2}(\mathbb{S}^{n})}.
\end{split}\end{equation}
By (\ref{Funk-Heck}), we can write
\begin{align*}
\langle\mathcal{P}_{2s}(r_2),r_2\rangle=\sum_{j=0}^{+\infty}A_{n,s}(j)\|\pi_jr_2\|_{L^2(\mathbb{S}^n)}^2.
\end{align*}
This together with \eqref{int2} and the fact $A_{n,s}(K)$ is decreasing in $K$, gives that
\begin{align*}
&\int_{\mathbb{S}^{n}}r_2^{2}d\sigma_{\xi}-\langle\mathcal{P}_{2s}(r_2),r_2\rangle\\
 &\ \ \geq \sum_{j=K}^{+\infty}(1-A_{n,s}(j))\|\pi_jr_2\|_{L^2(\mathbb{S}^n)}^2\\
 &\ \ \geq (1-A_{n,s}(K))\left(\|r_2\|_{L^2(\mathbb{S}^n)}^2-\|\Pi_Kr_2\|_{L^{2}(\mathbb{S}^{n})}^2\right)\\
 &\ \ \geq (1-A_{n,s}(K))\left(1-3^{K}\gamma^{-\frac{p}{2}}\delta_0^{\frac{p}{4}}\right)\|r_2\|^2_{L^{2}(\mathbb{S}^{n})}.
\end{align*}
Then we conclude that
\begin{align*}
I_2&=\int_{\mathbb{S}^{n}}r_2^{2}d\sigma_{\xi}-\langle\mathcal{P}_{2s}(r_2),r_2\rangle\\
&\ \ +\left(p-2-\frac{2}{p}C_{\epsilon_1,\epsilon_2}(1+\vartheta)\theta-\sigma_0\theta\right)\int_{\mathbb{S}^{n}}r_2^{2}d\sigma_{\xi}\\
&\geq (1-A_{n,s}(K))\left(1-3^{K}\gamma^{-\frac{p}{2}}\delta_0^{\frac{p}{4}}\right)\|r_2\|^2_{L^{2}(\mathbb{S}^{n})}\\
&\ \ -\frac{4s}{n+2s}\left(1+\frac{12}{7}C_{\epsilon_1, \epsilon_2}+\sigma_0\right)\|r_2\|_{L^{2}(\mathbb{S}^{n})}^2.
\end{align*}
Choose $\delta_2=\delta_2(K)$ satisfying
\begin{align}\label{delta2}
\left(1-3^{K}\gamma^{-\frac{p}{2}}\delta_2^{\frac{p}{4}}\right)= \frac{2}{3}.
\end{align}
 Noticing
\begin{equation}\begin{split}\nonumber
1-A_{n,s}(K)&\geq 1-A_{n,1}(K)\\
&=1-\frac{\Gamma(\frac{n}{2}+1)\Gamma(\frac{n}{2}-1+K)}{\Gamma(\frac{n}{2}-1)\Gamma(\frac{n}{2}+1+K)}\\
&=1-\frac{\frac{n}{2}(\frac{n}{2}-1)}{(\frac{n}{2}+K)(\frac{n}{2}+K-1)}\\
&=\frac{2K(2n+2K-2)}{(n+2K)(n+2K-2)}
\geq \frac{2K}{n+2K},
\end{split}\end{equation}
we deduce that there exist $K_0$ and $n_0$ dependent on $s$ and $\epsilon_1$, $\epsilon_2$ such that
$$\frac{2}{3}\frac{2K_0}{n_0+2K_0}\geq \frac{4s}{n_0+2s}\left(1+\frac{12}{7}C_{\epsilon_1, \epsilon_2}+\sigma_0\right).$$
It is easy to check that $$\left(\frac{2K_0}{n+2K_0}\right)\left(\frac{4s}{n+2s}\right)^{-1}\geq \left(\frac{2K_0}{n_0+2K_0}\right)\left(\frac{4s}{n_0+2s}\right)^{-1}$$ when $n\geq n_0$. Choose $K=K_0$, when $n\geq n_0$, there holds
$$\frac{2}{3}\left(1-A_{n,s}(K)\right)\geq \frac{4s}{n+2s}\left(1+\frac{12}{7}C_{\epsilon_1, \epsilon_2}+\sigma_0\right).$$

To summarize, if we choose
\begin{align}\label{delta0}
\delta_0=\min\{\delta_1,\delta_2\},
\end{align}
 where $\delta_1$
 and $\delta_2$ are defined by (\ref{delta1}) and (\ref{delta2}) which are independent of $n$,  we can
prove $I_i\geq 0$ for $i=1,2,3$, which implies Lemma \ref{local sta}.

\section{Stability of HLS inequalities with the optimal asymptotic lower bounds}
In this section, we will deduce the global HLS stability with the optimal asymptotic lower bound from the local stability. The proof is based on the competing symmetries method, the rearrangement flow for HLS integral and concavity argument.

For any $0\leq g\in L^{\frac{2n}{n+2s}}$, let $\phi\in \mathcal{M}_{HLS}$ such that
$$\|(-\Delta)^{-s/2}(g-\phi)\|_2=\inf_{h\in\mathcal{M}_{HLS}}\|(-\Delta)^{-s/2}(g-h)\|_2.$$
Since the extremal of $\inf_{h\in \mathcal{M}_{HLS}}\|(-\Delta)^{-s/2}(g-h)\|$ may not be unique, then the norm $\|g-\phi\|_{\frac{2n}{n+2s}}$ is not unique.
First we prove the norms $\|g-\phi\|_{\frac{2n}{n+2s}}$ can be controlled by each other.

\begin{lemma}\label{comparable}
Let  $0\leq g\in L^{\frac{2n}{n+2s}}(\mathbb{R}^n)$ with $g=\phi_{i}+r_{i}~~(i=1,2)$, where $\phi_i\in \mathcal{M}_{HLS}$ such that
$\|(-\Delta)^{-s/2}(g-\phi_i)\|_2=\inf_{h\in\mathcal{M}_{HLS}}\|(-\Delta)^{-s/2}(g-h)\|_2.$ Then there holds
$$b_{n,s}^{-1}\|r_1\|_{\frac{2n}{n+2s}} \leq \|r_2\|_{\frac{2n}{n+2s}}\leq b_{n,s}\|r_1\|_{\frac{2n}{n+2s}},$$
where $b_{n,s}=1+2^{1+\frac{s}{n}}\sqrt{\frac{n+2s}{n-2s}}$.
\end{lemma}

\begin{proof}
Without loss of generality, we may assume that  $\|\phi_1\|_{\frac{2n}{n+2s}}=\|\phi_2\|_{\frac{2n}{n+2s}}=1$.
Since $\|(-\Delta)^{-s/2}(g-\phi_1)\|^2_2=\|(-\Delta)^{-s/2}(g-\phi_2)\|^2_2$, then we have
$$\|(-\Delta)^{-s/2}(\phi_1-\phi_2)\|^2_2=2|\langle(-\Delta)^{-s/2}(g-\phi_2),(-\Delta)^{-s/2}(\phi_1-\phi_2)\rangle|.$$
By the H$\ddot{o}$lder inequality and the HLS inequality, there holds
$$\|(-\Delta)^{-s/2}(\phi_1-\phi_2)\|_2\leq 2\|(-\Delta)^{-s/2}(g-\phi_2)\|_2\leq 2\mathcal{S}_{s,n}^{-1/2}\|r_2\|_{\frac{2n}{n+2s}}.$$
Using the same estimate we also get
$$\|(-\Delta)^{-s/2}(\phi_1-\phi_2)\|_2\leq 2\mathcal{S}_{s,n}^{-1/2}\|r_1\|_{\frac{2n}{n+2s}}.$$
We claim that \begin{equation}\label{reverse}
\|\phi_1-\phi_2\|_{\frac{2n}{n+2s}}\leq \sqrt{\frac{n+2s}{n-2s}}2^{s/n}\mathcal{S}_{s,n}^{1/2}\|(-\Delta)^{-s/2}(\phi_1-\phi_2)\|_2.
\end{equation}
 By this claim and the estimates above, we can derive
$$\|\phi_1-\phi_2\|_{\frac{2n}{n+2s}}\leq 2^{1+\frac{s}{n}}\sqrt{\frac{n+2s}{n-2s}}\|r_i\|_{\frac{2n}{n+2s}},~~i=1,2.$$
This along with the estimate
$$ \frac{\|r_2\|_{\frac{2n}{n+2s}}}{\|r_1\|_{\frac{2n}{n+2s}}}\leq \frac{\|r_1\|_{\frac{2n}{n+2s}}+\|\phi_1-\phi_2\|_{\frac{2n}{n+2s}}}{\|r_1\|_{\frac{2n}{n+2s}}}$$
and
$$\frac{\|r_1\|_{\frac{2n}{n+2s}}}{\|r_2\|_{\frac{2n}{n+2s}}}\leq \frac{\|r_2\|_{\frac{2n}{n+2s}}+\|\phi_1-\phi_2\|_{\frac{2n}{n+2s}}}{\|r_2\|_{\frac{2n}{n+2s}}}$$
yields that
$$\left(1+2^{1+\frac{s}{n}}\sqrt{\frac{n+2s}{n-2s}}\right)^{-1}\leq\frac{\|r_2\|_{\frac{2n}{n+2s}}}{\|r_1\|_{\frac{2n}{n+2s}}}\leq 1+2^{1+\frac{s}{n}}\sqrt{\frac{n+2s}{n-2s}},$$
which means
$b_{n,s}^{-1}\leq\frac{\|r_2\|_{\frac{2n}{n+2s}}}{\|r_1\|_{\frac{2n}{n+2s}}}\leq b_{n,s}$.

Now let us prove inequality \eqref{reverse}.
Since $\phi_i\geq 0$ and satisfy the Euler-Lagrange equation
$(-\Delta)^{-s}\phi_i=\mathcal{S}_{s,n}^{-1}\phi_i^{\frac{n-2s}{n+2s}}$. So $$\mathcal{S}_{s,n}\|\phi_1\|^2_{\frac{2n}{n+2s}}=\mathcal{S}_{s,n}\|\phi_2\|^2_{\frac{2n}{n+2s}}=\|(-\Delta)^{-s/2}g\|_2^2-\inf_{h\in\mathcal{M}_{HLS}}\|(-\Delta)^{-s/2}(g-h)\|^2_2.$$

Using the inequality $a^q-b^q\geq q a^{q-1}(a-b)$ when $a\geq b>0,~~0<q<1$, and H$\ddot{o}$lder's inequality for index less than 1, we have
 \begin{equation}\begin{split}\nonumber
 &\mathcal{S}_{s,n}\|(-\Delta)^{-s/2}(\phi_1-\phi_2)\|^2_2=\langle \phi_1^{\frac{n-2s}{n+2s}}-\phi_2^{\frac{n-2s}{n+2s}}, \phi_1-\phi_2\rangle\\
 &\ \ \geq \frac{n-2s}{n+2s}\int_{\{\phi_1\geq \phi_2\}} \phi_1^{\frac{-4s}{n+2s}}(\phi_1-\phi_2)^2dx+ \frac{n-2s}{n+2s}\int_{\{\phi_1\leq  \phi_2\}} \phi_2^{\frac{2n}{n+2s}-2}(\phi_1-\phi_2)^2dx\\
 &\ \ \geq \frac{n-2s}{n+2s}\left(\int_{\{\phi_1\geq \phi_2\}} |\phi_1|^{\frac{2n}{n+2s}}\right)^{-\frac{2s}{n}}\|(\phi_1-\phi_2)\chi_{\{\phi_1\geq \phi_2\}}\|_{\frac{2n}{n+2s}}^2\\
 &+\frac{n-2s}{n+2s}\left(\int_{\{\phi_1\leq \phi_2\}} |\phi_2|^{\frac{2n}{n+2s}}\right)^{-\frac{2s}{n}}\|(\phi_1-\phi_2)\chi_{\{\phi_1\leq \phi_2\}}\|_{\frac{2n}{n+2s}}^2\\
 &\ \ \geq\frac{n-2s}{n+2s}2^{-2s/n}\|\phi_1-\phi_2\|_{\frac{2n}{n+2s}}^2,
 \end{split}\end{equation}
which is the claim.

\end{proof}

Then let us state some tools including competing symmetry
theorem, a continuous rearrangement flow, which will be used in our proof. Let $f\in L^{p}(\mathbb{R}^n)$, $1<p<\infty$ and $f_k=(\mathcal{R}U)^kf,~~k\in \mathbb{N}$, where $\mathcal{R}f=f^\ast$ is the decreasing rearrangement of $f$ and
$$(Uf)(x)=\left(\frac{2}{|x-e_n|^2}\right)^{\frac{n}{p}}f\left(\frac{x_1}{|x-e_n|^2},\cdots,\frac{x_{n-1}}{|x-e_n|^2},\frac{|x|^2-1}{|x-e_n|^2}\right),$$
$e_n=(0,\cdots,0,1)\in \mathbb{R}^n$. Carlen and Loss~\cite{CaLo} proved the following competing symmetry theorem.

\begin{theorem}\label{compete sym}(Carlen-Loss)
Let $0\leq f\in L^{p}$ ($1<p<\infty$), $f_k=(\mathcal{R}U)^kf$ and $h(x)=\|f\|_{p}|\mathbb{S}^n|^{-1/p}\left(\frac{2}{1+|x|^2}\right)^{n/p}$.  Then
$$\lim_{k\rightarrow \infty}\|f_k(x)-h(x)\|_{p}=0.$$
\end{theorem}
A continuous rearrangement flow which interpolates between a function and its symmetric decreasing rearrangement introduced in \cite{DEFFL} based on Brock's flow (\cite{Br1}, \cite{Br2}) plays an important part in our proof. More specifically, there
exists a flow $f_\tau$, $\tau\in [0,\infty]$, such that
$$f_0=f,~~f_\infty=f^\ast.$$
And if $0\leq f\in L^p(\mathbb{R}^n)$ for some $1\leq p<\infty$, then $\tau\rightarrow f_{\tau}$ is continuous in $L^{p}(\mathbb{R}^n)$ and $\|f_\tau\|_{p}=\|f\|_{p}$.
For nonnegative functions $f,g$, the following Riesz's inequality for continuous convex rearrangement holds(see Appendix A in \cite{DEFFL}):
  $$\iint_{\mathbb{R}^n\times \mathbb{R}^n}\frac{f_\tau(x)g_\tau(y)}{|x-y|^{n-2s}}dxdy\geq \iint_{\mathbb{R}^n\times \mathbb{R}^n}\frac{f(x)g(y)}{|x-y|^{n-2s}}dxdy.$$

The continuity(or semi-continuity)of the rearrangement flow under different norm is very important in the proof of global stability from local stability. For instance,   Dolbeault, Esteban,
Figalli, Frank and Loss \cite{DEFFL} proved that $\tau\rightarrow\|\nabla f_{\tau}\|_2$ is a right-continuous function and the current authors \cite{CLT2} proved $\tau\rightarrow \inf_{h\in M_{HLS}} \|f_{\tau}-h\|_{\frac{2n}{n+2s}}$ is
a continuous function.

In our setting, the situation is more complicated. For any $g_\tau$, let us recall $g_\tau$ has a decomposition means $g_\tau=\phi_\tau+r_\tau$ with $\phi_\tau\in \mathcal{M}_{HLS}$ satisfying $\|(-\Delta)^{-s/2}(g_\tau-\phi_\tau)\|_2=\inf_{h\in\mathcal{M}_{HLS}}\|(-\Delta)^{-s/2}(g_\tau-h)\|$. Since the decomposition may not be unique, then $\|r_\tau\|_{\frac{2n}{n+2s}}$ is not well defined, although $\|(-\Delta)^{-s/2}r_\tau\|_2$ is well defined. So we need to understand the continuity $\tau\rightarrow\|r_{\tau}\|_{\frac{2n}{n+2s}}$ in a reasonable way.

\subsection{Continuity}

In this section we first prove two important lemmas. First we prove that although $g_k$ may not have the unique decomposition $g_k=\phi_k+r_k$, $\|r_k\|_{\frac{2n}{n+2s}}$ will be small enough for any decomposition of $g_k$ when $k$  is sufficiently big.

\begin{lemma}\label{strong}
For any $0\leq g \in L^{\frac{2n}{n+2s}}(\mathbb{R}^n)$, let $g_k=(\mathcal{R}U)^kg$. Assume that $g_k$ satisfies
$g_k=\phi_k+r_k$ with $\phi_k$ satisfying $\|(-\Delta)^{-\frac{s}{2}}(g_k-\phi_k)\|_{2}^2=\inf\limits_{h\in \mathcal{M}_{HLS}}\|(-\Delta)^{-\frac{s}{2}}(g_k-h)\|_{2}^2$, then $$\lim\limits_{k\rightarrow+\infty}\|r_k\|_{\frac{2n}{n+2s}}=0.$$
\end{lemma}

\begin{proof}
By Theorem \ref{compete sym}, we know that $$g\rightarrow  h_g(x)=\|g\|_{\frac{2n}{n+2s}}|\mathbb{S}^n|^{-\frac{n+2s}{2n}}\left(\frac{2}{1+|x|^2}\right)^{\frac{n+2s}{2}}\in M_{HLS}\ {\rm in}\ L^{\frac{2n}{n+2s}}(\mathbb{R}^n).$$
Then $$\|(-\Delta)^{-\frac{s}{2}}(g_k-\phi_k)\|_{2}^2\leq \|(-\Delta)^{-\frac{s}{2}}(g_k-h_g)\|_{2}^2\leq S_{n,s}\|g_k-h_g\|_{\frac{2n}{n+2s}}\rightarrow 0$$
as $k\rightarrow +\infty$ by the definition of $\phi_k$.

Obviously, $$\|r_k\|_{\frac{2n}{n+2s}}=\|g_k-\phi_k\|_{\frac{2n}{n+2s}}\leq \|g_k-h_g\|_{\frac{2n}{n+2s}}+\|\phi_k-h_g\|_{\frac{2n}{n+2s}}.$$
In order to prove that $\lim\limits_{k\rightarrow +\infty}\|r_k\|_{\frac{2n}{n+2s}}=0$, we only need to prove that
$$\lim\limits_{k\rightarrow +\infty}\|\phi_k-h_g\|_{\frac{2n}{n+2s}}=0.$$

We first claim that $$\lim\limits_{k\rightarrow +\infty}\|(-\Delta)^{-\frac{s}{2}}(\phi_k-h_g)\|_2=0.$$
Careful calculation yields that
\begin{equation}\begin{split}
\lim\limits_{k\rightarrow +\infty}\|(-\Delta)^{-\frac{s}{2}}(\phi_k-h_g)\|_2&\leq \lim\limits_{k\rightarrow +\infty}\|(-\Delta)^{-\frac{s}{2}}(\phi_k-g_k)\|_2+\lim\limits_{k\rightarrow +\infty}\|(-\Delta)^{-\frac{s}{2}}(g_k-h_g)\|_2\\
&\leq \lim\limits_{k\rightarrow +\infty}\|(-\Delta)^{-\frac{s}{2}}(h_g-g_k)\|_2+\lim\limits_{k\rightarrow +\infty}\|(-\Delta)^{-\frac{s}{2}}(g_k-h_g)\|_2\\
&\leq 2S_{n,s}^{\frac{1}{2}}\|g_k-h_g\|_{\frac{2n}{n+2s}}=0.
\end{split}\end{equation}
This implies that $$S_{n,s}\|\phi_k\|_{\frac{2n}{n+2s}}=\|(-\Delta)^{-\frac{s}{2}}\phi_k\|_2\rightarrow \|(-\Delta)^{-\frac{s}{2}}h_g\|_2=S_{n,s}\|h_g\|_{\frac{2n}{n+2s}}$$
as $k\rightarrow +\infty$. Now, we start to prove that $$\lim\limits_{k\rightarrow +\infty}\|\phi_k-h_g\|_{\frac{2n}{n+2s}}=0.$$
Since $\lim\limits_{k\rightarrow +\infty}\|(-\Delta)^{-\frac{s}{2}}(\phi_k-h_g)\|_2=0$, then it follows that
$$\lim\limits_{k\rightarrow +\infty} \langle(-\Delta)^{-s}\phi_k-(-\Delta)^{-s}h_g, \phi_k-h_g\rangle=0,$$
where $\langle\cdot,\cdot\rangle$ is the inner product in $L^{2}(\mathbb{R}^n)$. Noticing $\phi_k$ and $h_g$ are the extremals of Hardy-Littlewood-Sobolev inequality in $\mathbb{R}^n$, then we can deduce that there exist Lagrange multiplier $\lambda_k$ and $\lambda$ such that $$(-\Delta)^{-s}\phi_k=\lambda_k|\phi_k|^{\frac{2n}{n+2s}-2}\phi_k,\ \ (-\Delta)^{-s}h_g=\lambda|h_g|^{\frac{2n}{n+2s}-2}h_g.$$

Since $\lim\limits_{k\rightarrow +\infty}\|(-\Delta)^{-\frac{s}{2}}g_k\|_2= \|(-\Delta)^{-\frac{s}{2}}h_g\|_2$, then $\lim\limits_{k\rightarrow+\infty}\lambda_k=\lambda$. Careful calculation gives that
\begin{equation}\begin{split}\nonumber
0&=\lim\limits_{k\rightarrow +\infty} \langle(-\Delta)^{-s}\phi_k-(-\Delta)^{-s}h_g, \phi_k-h_g\rangle\\
&=\lim\limits_{k\rightarrow +\infty}\langle\lambda_k|\phi_k|^{\frac{2n}{n+2s}-2}\phi_k-\lambda|h_g|^{\frac{2n}{n+2s}-2}h_g, \phi_k-h_g\rangle\\
&=\lim\limits_{k\rightarrow +\infty}\langle\lambda_k-\lambda)|\phi_k|^{\frac{2n}{n+2s}-2}\phi_k, \phi_k-h_g\rangle\\
& +\lim\limits_{k\rightarrow +\infty}\langle\lambda|\phi_k|^{\frac{2n}{n+2s}-2}\phi_k-\lambda|h_g|^{\frac{2n}{n+2s}-2}h_g, \phi_k-h_g\rangle.
\end{split}\end{equation}
 This together with
 \begin{equation}\begin{split}\nonumber
 &\langle(\lambda_k-\lambda)|\phi_k|^{\frac{2n}{n+2s}-2}\phi_k, \phi_k-h_g\rangle\\
 &\lesssim (\lambda_k-\lambda) \|(-\Delta)^{-\frac{s}{2}}\phi_k\|_{2}\|(-\Delta)^{-\frac{s}{2}}(\phi_k-h_g)\|_{2}\rightarrow0
 \end{split}\end{equation}
 implies that
\begin{equation}\begin{split}\label{est 2}
\lim\limits_{k\rightarrow +\infty}\langle\lambda|\phi_k|^{\frac{2n}{n+2s}-2}\phi_k-\lambda|h_g|^{\frac{2n}{n+2s}-2}h_g, \phi_k-h_g\rangle=0.
 \end{split}\end{equation}

On the other hand, using the inequality $a^q-b^q\geq q a^{q-1}(a-b)$ when $a\geq b>0,~~0<q<1$, H$\ddot{o}$lder's inequality for index less than 1 and the fact $\phi_k$, $h_g> 0$, we have
 \footnotesize{\begin{equation}\begin{split}\label{est 3}
 &\langle|\phi_k|^{\frac{2n}{n+2s}-2}\phi_k-|h_g|^{\frac{2n}{n+2s}-2}h_g, \phi_k-h_g\rangle\\
 &\ \ \geq \frac{n-2s}{n+2s}\left(\int_{\{\phi_k\geq h_g\}} \phi_k^{\frac{2n}{n+2s}-2}(\phi_k-h_g)^2dx+ \int_{\{\phi_k\leq  h_g\}} h_g^{\frac{2n}{n+2s}-2}(h_g-\phi_k)^2dx\right)\\
 &\ \ \geq \frac{n-2s}{n+2s} \left(\int_{\{\phi_k\geq h_g\}} |\phi_k|^{\frac{2n}{n+2s}}\right)^{-\frac{2s}{n}}\|(\phi_k-h_g)\chi_{\{\phi_k\geq h_g\}}\|_{\frac{2n}{n+2s}}^2\\
 &\ \ \ \ \ +\frac{n-2s}{n+2s}\left(\int_{\{\phi_k\leq h_g\}} |h_g|^{\frac{2n}{n+2s}}\right)^{-\frac{2s}{n}}\|(\phi_k-h_g)\chi_{\{\phi_k\leq h_g\}}\|_{\frac{2n}{n+2s}}^2\\
 &\ \ \geq \min\left\{\left(\int_{\mathbb{R}^n} |\phi_k|^{\frac{2n}{n+2s}}\right)^{-\frac{2s}{n}},\  \left(\int_{\mathbb{R}^n} |h_g|^{\frac{2n}{n+2s}}\right)^{-\frac{2s}{n}}\right\} \frac{n-2s}{n+2s}2^{-2s/n}\|\phi_k-h_g\|_{\frac{2n}{n+2s}}^2.
 \end{split}\end{equation}}
 Combining the above estimates (\ref{est 2}) and (\ref{est 3}), we derive that $\lim\limits_{k\rightarrow +\infty}\|\phi_k-h_g\|_{\frac{2n}{n+2s}}=0$ and Lemma \ref{strong} is proved.

\end{proof}

Next we illustrate that we can understand the continuity of the rearrangement flow in the following way.

\begin{lemma}\label{continuity}
For $g_{\tau}\in L^{\frac{2n}{n+2s}}(\mathbb{R}^n)$ satisfying $g_\tau=\phi_{\tau}+r_{\tau}$ with
$\phi_{\tau}$ being any extremal of $\inf\limits_{h\in \mathcal{M}_{HLS}}\|(-\Delta)^{-\frac{s}{2}}(g_\tau-h)\|_{2}^2$. If $g_{\tau}$ converges to  $g_{\tau_0}\not\equiv 0$ in $L^{\frac{2n}{n+2s}}(\mathbb{R}^n)$, then there exists subsequence $\tau_k\rightarrow \tau_0$ (as $k\rightarrow +\infty$) and $r_{\tau_0}\in L^{\frac{2n}{n+2s}}(\mathbb{R}^n)$ such that
$r_{\tau_k}$ converges to $r_{\tau_0}$ in $L^{\frac{2n}{n+2s}}(\mathbb{R}^n)$, and $\phi_{\tau_0}:=g_{\tau_0}-r_{\tau_0}$ is an extremal of $\inf\limits_{h\in M_{HLS}}\|(-\Delta)^{-\frac{s}{2}}(g_{\tau_0}-h)\|_{2}^2$.
\end{lemma}

\subsection{The proof of Lemma \ref{continuity}}
We first prove that
\begin{equation}\label{weak-continuity}
\lim_{\tau\to\tau_0}\inf\limits_{h\in \mathcal{M}_{HLS}}\|(-\Delta)^{-\frac{s}{2}}(g_{\tau}-h)\|_{2}^2=\inf\limits_{h\in \mathcal{M}_{HLS}} \|(-\Delta)^{-\frac{s}{2}}(g_{\tau_0}-h)\|_{2}^2.
\end{equation}
Let $\widetilde{\phi}_{\tau_0}$ satisfy
$$\|(-\Delta)^{-\frac{s}{2}}(g_{\tau_0}-\widetilde{\phi}_{\tau_0})\|_{2}^2=
\inf\limits_{h\in \mathcal{M}_{HLS}}\|(-\Delta)^{-\frac{s}{2}}(g_{\tau_0}-h)\|_{2}^2.$$
Obviously, there holds
\begin{equation}\begin{split}\nonumber
&\|(-\Delta)^{-\frac{s}{2}}(g_\tau-\phi_\tau)\|_2-\|(-\Delta)^{-\frac{s}{2}}(g_{\tau_0}-\widetilde{\phi}_{\tau_0})\|_2\\
&\ \ \geq \|(-\Delta)^{-\frac{s}{2}}(g_\tau-\phi_\tau)\|_2-\|(-\Delta)^{-\frac{s}{2}}(g_{\tau_0}-\phi_{\tau})\|_2\\
&\ \ \geq -\|(-\Delta)^{-\frac{s}{2}}(g_\tau-g_{\tau_0})\|_2
\end{split}\end{equation}
 and
\begin{equation}\begin{split}\nonumber
&\|(-\Delta)^{-\frac{s}{2}}(g_{\tau_0}-\widetilde{\phi}_{\tau_0})\|_2-\|(-\Delta)^{-\frac{s}{2}}(g_{\tau}-\phi_{\tau})\|_2\\
&\ \ \geq \|(-\Delta)^{-\frac{s}{2}}(g_{\tau_0}-\widetilde{\phi}_{\tau_0})\|_2-\|(-\Delta)^{-\frac{s}{2}}(g_{\tau}-\widetilde{\phi}_{\tau_0})\|_2\\
&\ \ \geq -\|(-\Delta)^{-\frac{s}{2}}(g_\tau-g_{\tau_0})\|_2.
\end{split}\end{equation}
Combining the above estimates, we derive that
$$-\|(-\Delta)^{-\frac{s}{2}}(g_\tau-g_{\tau_0})\|_2\leq \|(-\Delta)^{-\frac{s}{2}}r_\tau\|_2-\|(-\Delta)^{-\frac{s}{2}}(g_{\tau_0}-\widetilde{\phi}_{\tau_0})\|_2\leq \|(-\Delta)^{-\frac{s}{2}}(g_\tau-g_{\tau_0})\|_2.$$
This together with the fact $\lim\limits_{\tau\rightarrow \tau_0}\|g_{\tau}-g_{\tau_0}\|_{\frac{2n}{n+2s}}=0$ and the HLS inequality implies that
\begin{equation}\label{weak}\begin{split}
\lim\limits_{\tau\rightarrow \tau_0}\int_{\mathbb{R}^n}|(-\Delta)^{-\frac{s}{2}}(g_\tau-\phi_{\tau})|^2dx=\int_{\mathbb{R}^n}|(-\Delta)^{-\frac{s}{2}}(g_{\tau_0}-\widetilde{\phi}_{\tau_0})|^2dx,
\end{split}\end{equation}
which is \eqref{weak-continuity}.
Observing $$\|(-\Delta)^{-\frac{s}{2}}g_{\tau}\|_2^2=\|(-\Delta)^{-\frac{s}{2}}\phi_{\tau}\|_2^2+\|(-\Delta)^{-\frac{s}{2}}(g_\tau-\phi_\tau)\|_2^2$$
and $$\|(-\Delta)^{-\frac{s}{2}}g_{\tau_0}\|_2^2=\|(-\Delta)^{-\frac{s}{2}}\widetilde{\phi}_{\tau_0}\|_2^2+\|(-\Delta)^{-\frac{s}{2}}(g_{\tau_0}-\widetilde{\phi}_{\tau_0})\|_2^2,$$
then using \eqref{weak}, $\lim\limits_{\tau\rightarrow \tau_0}\|g_{\tau}-g_{\tau_0}\|_{\frac{2n}{n+2s}}=0$ and HLS inequality, we deduce that
$$\lim\limits_{\tau\rightarrow \tau_0}\|(-\Delta)^{-\frac{s}{2}}\phi_{\tau}\|_2^2=\|(-\Delta)^{-\frac{s}{2}}\widetilde{\phi}_{\tau_0}\|_2^2.$$
Since $\phi_\tau$ and $\widetilde{\phi}_{\tau_0}$ are extremals of HLS inequality, hence
$$S_{n,s}\|(-\Delta)^{-\frac{s}{2}}\phi_{\tau}\|_2^2=\|\phi_{\tau}\|_{\frac{2n}{n+2s}},\ \ S_{n,s}\|(-\Delta)^{-\frac{s}{2}}\widetilde{\phi}_{\tau_0}\|_2^2=\|\widetilde{\phi}_{\tau_0}\|_{\frac{2n}{n+2s}}.$$
Then it follows that
\begin{equation}\begin{split}\label{convergence}
\lim\limits_{\tau\rightarrow \tau_0}\|\phi_{\tau}\|_{\frac{2n}{n+2s}}&=\lim\limits_{\tau\rightarrow \tau_0}S_{n,s}\|(-\Delta)^{-\frac{s}{2}}\phi_{\tau}\|_2^2\\
&=S_{n,s}\|(-\Delta)^{-\frac{s}{2}}\widetilde{\phi}_{\tau_0}\|_2^2=\|\widetilde{\phi}_{\tau_0}\|_{\frac{2n}{n+2s}}.
\end{split}\end{equation}
\vskip0.1cm

Now we are prepared to prove that there exists a subsequence $\tau_k\rightarrow \tau_0$ (as $k\rightarrow +\infty$) and $r_{\tau_0}\in L^{\frac{2n}{n+2s}}(\mathbb{R}^n)$ such that
$r_{\tau_k}$ converges to $r_{\tau_0}$ in $L^{\frac{2n}{n+2s}}(\mathbb{R}^n)$, where $\phi_{\tau_0}:=g_{\tau_0}-r_{\tau_0}$ is some extremal of $\inf\limits_{h\in \mathcal{M}_{HLS}}\|(-\Delta)^{-\frac{s}{2}}(g_{\tau_0}-h)\|_{2}^2$.
Assume that
$$\phi_{\tau}=c_{\tau}\frac{\lambda_{\tau}^{\frac{n+2s}{2}}}{(1+\lambda_\tau^2|x-x_\tau|^2)^{\frac{n+2s}{2}}}$$
with $c_{\tau}$, $\lambda_{\tau}$ being not equal to zero. Obviously, the limit $\lim\limits_{\tau\rightarrow \tau_0}c_{\tau}$ exist from (\ref{convergence}). Next, we prove that $\lambda_{\tau}$ is bounded. Suppose not, then there exists a subsequence $\tau_k\rightarrow \tau_0$ as $k\rightarrow +\infty$ such that $\lim\limits_{k\rightarrow +\infty}\lambda_{\tau_k}=+\infty$. Careful computation gives that
\begin{equation}\nonumber
\|(-\Delta)^{-\frac{s}{2}}r_{\tau_k}\|_{2}^2=\|(-\Delta)^{-\frac{s}{2}}g_{\tau_k}\|_{2}^2+\|(-\Delta)^{-\frac{s}{2}}\phi_{\tau_k}\|_{2}^2-2\int_{\mathbb{R}^n}
(-\Delta)^{-\frac{s}{2}}g_{\tau_k}(-\Delta)^{-\frac{s}{2}}\phi_{\tau_k}dx.
\end{equation}

We claim that if $\lambda_{\tau_k}\rightarrow +\infty$ as $k\rightarrow +\infty$, then
$$\lim\limits_{k\rightarrow +\infty}
\int_{\mathbb{R}^n}(-\Delta)^{-\frac{s}{2}}g_{\tau_k}(-\Delta)^{-\frac{s}{2}}\phi_{\tau_k}dx=0.$$
Combining this and \eqref{weak-continuity}, we derive that
\begin{equation}\begin{split}\nonumber
\inf\limits_{h\in M_{HLS}}\|(-\Delta)^{-\frac{s}{2}}(g_{\tau_0}-h)\|_{2}^2&=\lim\limits_{k\rightarrow+\infty}\|(-\Delta)^{-\frac{s}{2}}r_{\tau_k}\|_{2}^2\\
&=\lim\limits_{k\rightarrow +\infty}\left(\|(-\Delta)^{-\frac{s}{2}}g_{\tau_k}\|_{2}^2+\|(-\Delta)^{-\frac{s}{2}}\phi_{\tau_k}\|_{2}^2\right)\\
&=\|(-\Delta)^{-\frac{s}{2}}g_{\tau_0}\|_{2}^2+\|(-\Delta)^{-\frac{s}{2}}\widetilde{\phi}_{\tau_0}\|_{2}^2,
\end{split}\end{equation}
which is a contradiction with $$\|(-\Delta)^{-\frac{s}{2}}g_{\tau_0}\|_{2}^2=\inf\limits_{h\in M_{HLS}}\|(-\Delta)^{-\frac{s}{2}}(g_{\tau_0}-h)\|_{2}^2+\|(-\Delta)^{-\frac{s}{2}}\widetilde{\phi}_{\tau_0}\|_{2}^2$$
and the fact $\|(-\Delta)^{-\frac{s}{2}}\widetilde{\phi}_{\tau_0}\|_{2}\neq 0$. Hence in order to prove the boundedness of $\lambda_\tau$, we only need to prove that $$\lim\limits_{k\rightarrow +\infty}
\int_{\mathbb{R}^n}(-\Delta)^{-\frac{s}{2}}g_{\tau_k}(-\Delta)^{-\frac{s}{2}}\phi_{\tau_k}dx=0$$
if $\lambda_{\tau_k}\rightarrow +\infty$.
We distinguish this by two cases.

\vskip 0.1cm

\emph{Case 1:} $\{x_{\lambda_k}\}_k$ is bounded. With loss of generality, we may assume that $x_{\lambda_k}$ is equal to zero. First it is easy to show that
 that for any $\epsilon>0$, $\lim\limits_{k\rightarrow \infty}\int_{\mathbb{R}^n\setminus B_{\epsilon}}|\phi_{\tau_k}|^{\frac{2n}{n+2s}}dx=0$ since $\lambda_{\tau_k}\rightarrow \infty$ as $k\rightarrow \infty$.
Direct computation yields that
\begin{equation}\begin{split}\nonumber
\int_{\mathbb{R}^n}(-\Delta)^{-\frac{s}{2}}g_{\tau_k}(-\Delta)^{-\frac{s}{2}}\phi_{\tau_k}dx&=\int_{\mathbb{R}^n}\int_{\mathbb{R}^n\setminus B_{\epsilon}}
\frac{\phi_{\tau_k}(y)}{|x-y|^{n-2s}}dy g_{\tau_k}(x)dx\\
&+\int_{\mathbb{R}^n}\int_{B_{\epsilon}}
\frac{\phi_{\tau_k}(y)}{|x-y|^{n-2s}}dy g_{\tau_k}(x)dx.\\
\end{split}\end{equation}
Through HLS inequality, we have $$\lim\limits_{k\rightarrow \infty}\int_{\mathbb{R}^n}\int_{\mathbb{R}^n\setminus B_{\epsilon}}
\frac{\phi_{\tau_k}(y)}{|x-y|^{n-2s}}dy g_{{\tau_k}}(x)dx\lesssim \lim\limits_{k\rightarrow \infty}\left(\int_{\mathbb{R}^n\setminus B_{\epsilon}}|\phi_{{\tau_k}}|^{\frac{2n}{n+2s}}dx\right)^{\frac{n+2s}{2n}}\|g_{\tau_k}\|_{\frac{2n}{n+2s}}=0.$$
Next, we prove $\lim\limits_{k\rightarrow \infty}\int_{\mathbb{R}^n}\int_{B_{\epsilon}}
\frac{\phi_{\tau_k}(y)}{|x-y|^{n-2s}}dy g_{{\tau_k}}(x)dx=0$. We split this double integral into two parts:
\begin{equation}\begin{split}\nonumber
&\int_{\mathbb{R}^n}\int_{B_{\epsilon}}
\frac{\phi_{\tau_k}(y)}{|x-y|^{n-2s}}dy g_{{\tau_k}}(x)dx\\
&=\int_{B_{2\epsilon}}\int_{B_{\epsilon}}
\frac{\phi_{\tau_k}(y)}{|x-y|^{n-2s}}dy g_{{\tau_k}}(x)dx+\int_{\mathbb{R}^n\setminus B_{2\epsilon}}\int_{B_{\epsilon}}
\frac{\phi_{\tau_k}(y)}{|x-y|^{n-2s}}dy g_{{\tau_k}}(x)dx\\
&=I+II.
\end{split}\end{equation}
For $I$, through HLS inequality and $g_{{\tau_k}}\rightarrow g_{\tau_0}$ in $L^{\frac{2n}{n+2s}}(\mathbb{R}^n)$, we obtain that
\begin{equation}\begin{split}\nonumber
& \lim\limits_{k\rightarrow \infty}\lim\limits_{k\rightarrow \infty}\int_{B_{2\epsilon}}\int_{B_{\epsilon}}
\frac{\phi_{\tau_k}(y)}{|x-y|^{n-2s}}dy g_{{\tau_k}}(x)dx\\
& \lesssim\lim\limits_{k\rightarrow \infty}\lim\limits_{\tau\rightarrow \tau_0}\left(\int_{B_{2\epsilon}}|g_{\tau_k}|^{\frac{2n}{n+2s}}\right)^{\frac{2n}{n+2s}}\|\phi_{\tau_k}\|_{\frac{2n}{n+2s}}=0.
\end{split}\end{equation}
For $II$, we can write
\begin{equation}\begin{split}\nonumber
\int_{\mathbb{R}^n\setminus B_{2\epsilon}}\int_{B_{\epsilon}}
\frac{\phi_{\tau_k}(y)}{|x-y|^{n-2s}}dy g_{{\tau_k}}(x)dx&\leq \frac{1}{\epsilon^{\kappa}}\int_{\mathbb{R}^n\setminus B_{2\epsilon}}\int_{B_{\epsilon}}
\frac{\phi_{\tau_k}(y)}{|x-y|^{n-2s-\kappa}}dy g_{{\tau_k}}(x)dx\\
&\lesssim  \frac{1}{\epsilon^{\kappa}}\|g_{{\tau_k}}\|_{\frac{2n}{n+2s}}\left(\int_{B_{\epsilon}}|\phi_{\tau_k}|^{q}dx\right)^{\frac{1}{q}},
\end{split}\end{equation}
where $\kappa$ is sufficiently small and $q$ satisfying $\frac{1}{q}+\frac{n+2s}{2n}+\frac{n-2s-\kappa}{n}=2$.
Obviously we have $0<q<\frac{2n}{n+2s}$. Since $|\phi_{\tau_k}|^{\frac{2n}{n+2s}}\rightharpoonup \|\phi_{\tau_0}\|_{\frac{2n}{n+2s}}^{\frac{2n}{n+2s}}\delta_0$, then it follows that for any $0<q<\frac{2n}{n+2s}$ and bounded domain $\Omega$ containing origin, there holds
$$\lim\limits_{k\rightarrow \infty}\int_{\Omega}|\phi_{\tau_k}|^{q}dx=0.$$ This yields that
$$\lim\limits_{k\rightarrow \infty}II=0.$$ In conclusion, we finish the proof of $$\lim\limits_{k\rightarrow \infty}
\int_{\mathbb{R}^n}(-\Delta)^{-\frac{s}{2}}g_{{\tau_k}}(-\Delta)^{-\frac{s}{2}}\phi_{{\tau_k}}dx=0$$ when $\lambda_{\tau_k}\rightarrow \infty$ in the case of $\{x_{{\tau_k}}\}$ being bounded.
\medskip

\emph{Case 2:} $\{x_{{\tau_k}}\}_{k}$ is unbounded. We assume that $|x_{{\tau_k}}|\rightarrow +\infty$. Careful computation gives that
\begin{equation}\begin{split}\nonumber
&\int_{\mathbb{R}^n}(-\Delta)^{-\frac{s}{2}}g_{{\tau_k}}(-\Delta)^{-\frac{s}{2}}\phi_{{\tau_k}}dx=\int_{\mathbb{R}^n}\int_{\mathbb{R}^n}
\frac{\phi_{\tau_k}(y+x_{\tau_k})}{|x-y|^{n-2s}}dy g_{{\tau_k}}(x+x_{\tau_k})dx\\
&=\int_{\mathbb{R}^n}\int_{B_{\epsilon}}
\frac{\phi_{\tau_k}(y+x_{\tau_k})}{|x-y|^{n-2s}}dy g_{{\tau_k}}(x+x_{\tau_k})dx+\int_{\mathbb{R}^n}\int_{\mathbb{R}^n\setminus B_{\epsilon}}
\frac{\phi_{\tau_k}(y+x_{\tau_k})}{|x-y|^{n-2s}}dy g_{{\tau_k}}(x+x_{\tau_k})dx.
\end{split}\end{equation}
 As what we did in \emph{Case 1}, we similarly derive that
\begin{equation*}\begin{split}
&\int_{\mathbb{R}^n}\int_{\mathbb{R}^n\setminus B_{\epsilon}}
\frac{\phi_{\tau_k}(y+x_{\tau_k})}{|x-y|^{n-2s}}dy g_{{\tau_k}}(x+x_{{\tau_k}})dx\\
& \lesssim \lim\limits_{k\rightarrow \infty}\left(\int_{\mathbb{R}^n\setminus B_{\epsilon}}|\phi_{{\tau_k}}(x+x_{\tau_k})|^{\frac{2n}{n+2s}}dx\right)^{\frac{n+2s}{2n}}\|g_{\tau_k}\|_{\frac{2n}{n+2s}}=0,
\end{split}\end{equation*}
since $|\phi_{\tau_k}(x+x_{\tau_k})|^{\frac{2n}{n+2s}}\rightharpoonup \|\phi_{\tau_0}\|_{\frac{2n}{n+2s}}^{\frac{2n}{n+2s}}\delta_0$.

For $\int_{\mathbb{R}^n}\int_{\mathbb{R}^n\setminus B_{\epsilon}}
\frac{\phi_{\tau_k}(y+x_{\tau_k})}{|x-y|^{n-2s}}dy g_{{\tau_k}}(x+x_{\tau_k})dx$, we can write
\begin{equation}\begin{split}\nonumber
&\int_{\mathbb{R}^n}\int_{B_{\epsilon}}
\frac{\phi_{\tau_k}(y+x_{{\tau_k}})}{|x-y|^{n-2s}}dy g_{{\tau_k}}(x+x_{{\tau_k}})dx\\
&\ \ =\int_{B_R}\int_{B_{\epsilon}}
\frac{\phi_{\tau_k}(y+x_{{\tau_k}})}{|x-y|^{n-2s}}dy g_{{\tau_k}}(x+x_{{\tau_k}})dx+\int_{\mathbb{R}^n\setminus B_{R}}\int_{B_{\epsilon}}
\frac{\phi_{\tau_k}(y+x_{{\tau_k}})}{|x-y|^{n-2s}}dy g_{{\tau_k}}(x+x_{{\tau_k}})dx\\
&\ \ =I+II.
\end{split}\end{equation}
For $I$, through HLS inequality, we immediately derive that
$$\int_{B_R}\int_{B_{\epsilon}}
\frac{\phi_{\tau_k}(y+x_{{\tau_k}})}{|x-y|^{n-2s}}dy g_{{\tau_k}}(x+x_{{\tau_k}})dx\lesssim \left(\int_{B_{R}}|g_{\tau_k}(x+x_{{\tau_k}})|^{\frac{2n}{n+2s}}dx\right)^{\frac{2n}{n+2s}}\|\phi_{\tau_k}\|_{\frac{2n}{n+2s}}=0$$
as $k\rightarrow \infty$ by the fact
$$\lim\limits_{k\rightarrow \infty}\int_{B_{R}}|g_{\tau_k}(x+x_{{\tau_k}})|^{\frac{2n}{n+2s}}\leq \lim\limits_{k\rightarrow \infty}\int_{\{|x|\geq |x_{{\tau_k}}|-R\}}|g_{\tau_k}(x)|^{\frac{2n}{n+2s}}=0,$$
since $g_{\tau_k}$ converges strongly to $g_{\tau_0}$ in $L^{\frac{2n}{n+2s}}(\mathbb{R}^n)$ and $\lambda_{{\tau_k}}\rightarrow \infty$. For $II$, we can write
\begin{equation}\begin{split}\nonumber
&\int_{\mathbb{R}^n\setminus B_{R}}\int_{B_{\epsilon}}
\frac{\phi_{\tau_k}(y+x_{{\tau_k}})}{|x-y|^{n-2s}}dy g_{{\tau_k}}(x+x_{{\tau_k}})dx\\
&\ \ \leq \frac{1}{(R-\epsilon)^{\kappa}}\int_{\mathbb{R}^n\setminus B_{R}}\int_{B_{\epsilon}}
\frac{\phi_{\tau_k}(y+x_{\tau_k})}{|x-y|^{n-2s-\kappa}}dy g_{{\tau_k}}(x+x_{\tau_k})dx\\
&\ \ \lesssim  \frac{1}{(R-\epsilon)^{\kappa}}\|g_{{\tau_k}}\|_{\frac{2n}{n+2s}}\left(\int_{B_{\epsilon}}|\phi_{\tau_k}(x+x_{{\tau_k}})|^{q}dx\right)^{\frac{1}{q}},
\end{split}\end{equation}
where $\kappa$ is sufficiently small and $q$ satisfying $\frac{1}{q}+\frac{n+2s}{2n}+\frac{n-2s-\kappa}{n}=2$.
 As what we did in \emph{Case 1}, we can get
$$\lim\limits_{k\rightarrow \infty}\int_{B_{\epsilon}}|\phi_{\tau_k}(x+x_{{\tau_k}})|^{q}dx=0.$$ This yields that
$$\lim\limits_{k\rightarrow \infty}II=0.$$ In conclusion, we finish the proof of $$\lim\limits_{k\rightarrow \infty}
\int_{\mathbb{R}^n}(-\Delta)^{-\frac{s}{2}}g_{{\tau_k}}(-\Delta)^{-\frac{s}{2}}\phi_{{\tau_k}}dx=0$$ in the case of $\{x_{{\tau_k}}\}$ being unbounded.
\medskip

Combining the Case 1 and Case 2, we finish the proof of boundedness of $\{\lambda_{\tau}\}$.
\medskip

Next, we claim that $\lambda_{\tau}$ must have the positive lower bound. We argue this by contradiction. Suppose not, there exists a subsequence $\tau_k\rightarrow \tau_0$ such that $\lambda_{\tau_k}\rightarrow 0$ as
$k\rightarrow \infty$, then $\phi_{{\tau_k}}$ strongly converges to zero in $L^{\infty}(\mathbb{R}^n)$. It is also not difficult to check that $\phi_{{\tau_k}}\rightharpoonup 0$ weakly in $L^{\frac{2n}{n+2s}}(\mathbb{R}^n)$ and $|\phi_{{\tau_k}}|^{\frac{2n}{n+2s}-2}\phi_{{\tau_k}}\rightharpoonup 0$ weakly in $L^{\frac{2n}{n-2s}}(\mathbb{R}^n)$. Since $g_{\tau_0}\in L^{\frac{2n}{n+2s}}(\mathbb{R}^n)$, by the property of weak convergence of $|\phi_{{\tau_k}}|^{\frac{2n}{n+2s}}\phi_{{\tau_k}}$, we immediately
derive that
\begin{equation}\nonumber
\lim\limits_{k\rightarrow \infty}\int_{\mathbb{R}^n}g_{\tau_0}|\phi_{{\tau_k}}|^{\frac{2n}{n+2s}-2}\phi_{{\tau_k}}dx=0,
\end{equation}
which implies that
\begin{equation}\label{imp1}
\lim\limits_{k\rightarrow \infty}\int_{\mathbb{R}^n}(-\Delta)^{-\frac{s}{2}}g_{\tau_0}(-\Delta)^{-\frac{s}{2}}\phi_{{\tau_k}}dx=0
\end{equation}
since $\phi_{{\tau_k}}$ is the extremal function of HLS inequality and satisfies the Euler-Lagrange Equation $(-\Delta)^{-s}\phi=\mathcal{S}_{s,n}|\phi|^{-\frac{4s}{n+2s}}\phi$. On the other hand, since $g_{{\tau_k}}$ strongly converges to $g_{\tau_0}$ in $L^{\frac{2n}{n+2s}}(\mathbb{R}^n)$, we derive that
\begin{equation}\label{imp2}\begin{split}
\lim\limits_{k\rightarrow \infty}\int_{\mathbb{R}^n}(-\Delta)^{-\frac{s}{2}}(g_{{\tau_k}}-g_{\tau_0})(-\Delta)^{-\frac{s}{2}}\phi_{{\tau_k}}dx&\lesssim \lim\limits_{k\rightarrow \infty}\|g_{{\tau_k}}-g_{\tau_0}\|_{\frac{2n}{n+2s}}\|\phi_{{\tau_k}}\|_{\frac{2n}{n+2s}}^{\frac{n-2s}{n+2s}}=0.
\end{split}\end{equation}
Combining \eqref{imp1} and \eqref{imp2}, we conclude that
$$\lim\limits_{k\rightarrow \infty}\int_{\mathbb{R}^n}(-\Delta)^{-\frac{s}{2}}g_{{\tau_k}}(-\Delta)^{-\frac{s}{2}}\phi_{{\tau_k}}dx=0.$$
Then it follows that \begin{equation}\begin{split}\nonumber
\|(-\Delta)^{-\frac{s}{2}}r_{\tau_0}\|_{2}^2&=\lim\limits_{k\rightarrow \infty}\|(-\Delta)^{-\frac{s}{2}}r_{{\tau_k}}\|_{2}^2\\
&=\lim\limits_{k\rightarrow \infty}\left(\|(-\Delta)^{-\frac{s}{2}}g_{{\tau_k}}\|_{2}^2+\|(-\Delta)^{-\frac{s}{2}}\phi_{{\tau_k}}\|_{2}^2\right)\\
&=\|(-\Delta)^{-\frac{s}{2}}g_{\tau_0}\|_{2}^2+\|(-\Delta)^{-\frac{s}{2}}\phi_{\tau_0}\|_{2}^2,
\end{split}\end{equation}
which is a contradiction with $$\|(-\Delta)^{-\frac{s}{2}}g_{\tau_0}\|_{2}^2=\|(-\Delta)^{-\frac{s}{2}}r_{\tau_0}\|_{2}^2+\|(-\Delta)^{-\frac{s}{2}}\phi_{\tau_0}\|_{2}^2,$$
and the fact $\|(-\Delta)^{-\frac{s}{2}}\phi_{\tau_0}\|_{2}\neq 0$. This proves that $\lambda_{{\tau_k}}$ has the positive lower bound.

\medskip

Next we start to show that $\{x_{{\tau_k}}\}_{k}$ is also bounded. We argue this by contradiction again. Suppose not, there exists subsequence $\tau_k\rightarrow \tau_0$ such that $$\lim\limits_{k\rightarrow +\infty}|x_{\tau_k}|=+\infty.$$ We will show
$$\lim\limits_{k\rightarrow +\infty}\int_{\mathbb{R}^n}(-\Delta)^{-\frac{s}{2}}g_{\tau_k}(-\Delta)^{-\frac{s}{2}}\phi_{\tau_k}dx=0.$$
 Since we have proved that $\lambda_{{\tau_k}}$ is bounded, if $|x_{\lambda_{\tau_k}}|\rightarrow +\infty$, then $\phi_{{\tau_k}}$ strongly converges to zero in $L^{\infty}_{loc}(\mathbb{R}^n)$ as $k\rightarrow \infty$. Write
\begin{equation}\begin{split}\nonumber
&\int_{\mathbb{R}^n}(-\Delta)^{-\frac{s}{2}}g_{{\tau_k}}(-\Delta)^{-\frac{s}{2}}\phi_{{\tau_k}}dx\\
&\ \ =\int_{\mathbb{R}^n}\int_{\mathbb{R}^n\setminus B_{2R}}
\frac{\phi_{\tau_k}(y)}{|x-y|^{n-2s}}dy g_{{\tau_k}}(x)dx+\int_{\mathbb{R}^n}\int_{B_{2R}}
\frac{\phi_{\tau_k}(y)}{|x-y|^{n-2s}}dy g_{{\tau_k}}(x)dx.\\
\end{split}\end{equation}
For $\int_{\mathbb{R}^n}\int_{B_{2R}}
\frac{\phi_{\tau_k}(y)}{|x-y|^{n-2s}}dy g_{{\tau_k}}(x)dx$, it holds that
$$\lim\limits_{k\rightarrow \infty}\int_{\mathbb{R}^n}\int_{B_{2R}}
\frac{\phi_{\tau_k}(y)}{|x-y|^{n-2s}}dy g_{{\tau_k}}(x)dx\lesssim \lim\limits_{k\rightarrow \infty}\|g_{{\tau_k}}\|_{\frac{2n}{n+2s}}\|\phi_{{\tau_k}}\|_{L^{\infty}(B_R)}|B_{2R}|^{\frac{n+2s}{2n}}=0.$$
For $\int_{\mathbb{R}^n}\int_{\mathbb{R}^n\setminus B_{2R}}
\frac{\phi_{\tau_k}(y)}{|x-y|^{n-2s}}dy g_{{\tau_k}}(x)dx$, we can write
\begin{equation}\begin{split}\nonumber
&\int_{\mathbb{R}^n}\int_{\mathbb{R}^n\setminus B_{2R}}
\frac{\phi_{\tau_k}(y)}{|x-y|^{n-2s}}dy g_{{\tau_k}}(x)dx\\
&\ \ =\int_{\mathbb{R}^n\setminus B_{R}}\int_{\mathbb{R}^n\setminus B_{2R}}
\frac{\phi_{\tau_k}(y)}{|x-y|^{n-2s}}dy g_{{\tau_k}}(x)dx+\int_{B_{R}}\int_{\mathbb{R}^n\setminus B_{2R}}
\frac{\phi_{\tau_k}(y)}{|x-y|^{n-2s}}dy g_{{\tau_k}}(x)dx\\
&=I+II.
\end{split}\end{equation}
For $I$, we immediately have
\begin{equation}\begin{split}\nonumber
&\lim\limits_{R\rightarrow+\infty}\lim\limits_{k\rightarrow \infty}\int_{\mathbb{R}^n\setminus B_{R}}\int_{\mathbb{R}^n\setminus B_{2R}}
\frac{\phi_{\tau_k}(y)}{|x-y|^{n-2s}}dy g_{{\tau_k}}(x)dx\\
&\ \ \lesssim \lim\limits_{R\rightarrow+\infty}\lim\limits_{k\rightarrow \infty}\left(\int_{\mathbb{R}^n\setminus B_{R}}|g_{{\tau_k}}|^{\frac{2n}{n+2s}}dx\right )^{\frac{n+2s}{n}}\|\phi_{{\tau_k}}\|_{\frac{2n}{n+2s}}\\
&\ \ =\lim\limits_{R\rightarrow+\infty}\left(\int_{\mathbb{R}^n\setminus B_{R}}|g_{\tau_0}|^{\frac{2n}{n+2s}}dx\right)^{\frac{n+2s}{n}}\|\phi_{\tau_0}\|_{\frac{2n}{n+2s}}=0.
\end{split}\end{equation}
For $II$, direct computation gives that
\begin{equation}\begin{split}\nonumber
&\int_{B_{R}}\int_{\mathbb{R}^n\setminus B_{2R}}
\frac{\phi_{\tau_k}(y)}{|x-y|^{n-2s}}dy g_{{\tau_k}}(x)dx\\
&\ \ \leq \frac{1}{(R)^{\kappa}}\int_{B_R}\int_{\mathbb{R}^n}
\frac{\phi_{\tau_k}(y)}{|x-y|^{n-2s-\kappa}}dy g_{{\tau_k}}(x)dx\\
&\ \ \lesssim  \frac{1}{R^{\kappa}}\|u_{{\tau_k}}\|_{\frac{2n}{n+2s}}\|\phi_{{\tau_k}}\|_{q},
\end{split}\end{equation}
where $\kappa$ is sufficiently small and $q$ satisfies $\frac{1}{q}+\frac{n+2s}{2n}+\frac{n-2s-\kappa}{n}=2$. Since we have already proved that $C_1\leq \lambda_{{\tau_k}}\leq C_2$, it is not difficult to check that $\int_{\mathbb{R}^n}|\phi_{{\tau_k}}|^qdx$ is bounded if $q$ is a little bigger than $\frac{2n}{n+2s}$.

Combining the above estimates, we conclude that
$$\lim\limits_{R\rightarrow +\infty}\lim\limits_{k\rightarrow\infty}\int_{B_{R}}\int_{\mathbb{R}^n\setminus B_{2R}}
\frac{\phi_{\tau_k}(y)}{|x-y|^{n-2s}}dy g_{{\tau_k}}(x)dx=0.$$ Thus, we have proved that $\{x_{{\tau_k}}\}_{k}$ is bounded.
\medskip

Since we have proved that $\{\lambda_{\tau}\}_{\tau}$ and $\{x_{\tau}\}_{\tau}$ are bounded, $\lambda_{\tau}$ has the positive lower bound,
$\lim\limits_{\tau\rightarrow \tau_0}c_{\tau}=c_{\tau_0}$, then there exist subsequence $\tau_k\rightarrow \tau_0$, $\lambda_{\tau_k}\rightarrow \lambda_{\tau_0}$, $x_{\tau_k}\rightarrow x_{\tau_0}$ as $k\rightarrow +\infty$ such that
$$\phi_{\tau_k}=c_{\tau_k} \frac{\lambda_{\tau_k}^{\frac{n+2s}{2}}}{\left(1+\lambda_{\tau_k}^2|x-x_{\tau_k}|^2\right)^{\frac{n+2s}{2}}}\rightarrow c_{\tau_0} \frac{\lambda_{\tau_0}^{\frac{n+2s}{2}}}{\left(1+\lambda_{\tau_0}^2|x-x_{\tau_0}|^2\right)^{\frac{n+2s}{2}}}:=\phi_{\tau_0},\ \ \forall\ x\in \mathbb{R}^n.$$
This together with Brezis-Lieb lemma implies that $\phi_{\tau_k}$ strongly converges to $\phi_{\tau_0}$ in $L^{\frac{2n}{n+2s}}(\mathbb{R}^n)$ as $k\rightarrow +\infty$.

Then we immediately have
\begin{equation}\begin{split}
\lim_{k\rightarrow \infty}\|(-\Delta)^{-\frac{s}{2}}(g_{\tau_k}-\phi_{\tau_k})\|_2=\|(-\Delta)^{-\frac{s}{2}}(g_{\tau_0}-\phi_{\tau_0})\|_2,
\end{split}\end{equation}
since $g_{\tau_k}-g_{\tau_0}\rightarrow 0$, $\phi_{\tau_k}-\phi_{\tau_0}\rightarrow 0$ in $L^{\frac{2n}{n+2s}}(\mathbb{R}^n)$. On the other hand, we have already proved that
$$\lim\limits_{k\rightarrow \infty}\int_{\mathbb{R}^n}|(-\Delta)^{-\frac{s}{2}}(g_{\tau_k}-\phi_{\tau_k})|^2dx=\inf\limits_{h\in M_{HLS}}\|(-\Delta)^{-\frac{s}{2}}(g_{\tau_0}-h)\|_{2}^2,$$
which means  $\phi_{\tau_0}$ is a extremals of $\inf\limits_{h\in M_{HLS}}\|(-\Delta)^{-\frac{s}{2}}(g_{\tau_0}-h)\|_{2}^2$.  Then, we accomplish the proof of Lemma \ref{continuity}.

\subsection{The stability of the HLS inequality}
Now we are in the position to prove the  stability of HLS inequality with the optimal asymptotic lower bound. In fact, by Lemma of \cite{CLT1}, a concavity argument can be adopted to connect the stability of nonnegative functions and general functions. Let us denote by $C_{HLS}$ the optimal constant for stability of the HLS inequality and denote by $C^{pos}_{HLS}$ the optimal constant for stability of the HLS inequality  when restricted to nonnegative functions. The relationship between these two optimal constants states

\begin{lemma}\label{constant}
$$C_{HLS}\geq \frac{1}{2}\min\{C^{pos}_{HLS},\min\{2^{\frac{n+2s}{n}}-2,1\}\}.$$
\end{lemma}

Thus we only need to consider the nonnegative function. Recall that we say $g$ has a decomposition means $g=\phi+r$ with $\phi$ satisfying $\phi\in \mathcal{M}_{HLS}$ and $$\|(-\Delta)^{-\frac{s}{2}}(g-\phi)\|_{2}=\inf\limits_{h\in M_{HLS}}\|(-\Delta)^{-\frac{s}{2}}(g-h)\|_{2}.$$
In Section 2  we have proved that if a nonnegative function $g\in L^{\frac{2n}{n+2s}}(\mathbb{R}^n)$ has a decomposition
$g=\phi+r$ with  $\|r\|_{\frac{2n}{n+2s}}^2\leq \delta \|g\|_{\frac{2n}{n+2s}}^2$, then
$$\left\| g \right\|_{\frac{2n}{n+2s}}^2-\mathcal S_{s,n} \|(-\Delta)^{-s/2}g\|_{2}^2\geq \frac{1}{n}C_s\inf_{h\in \mathcal{M}_{HLS}}\|g-h\|_{\frac{2n}{n+2s}}^2.$$

So we only need to handle the stability of HLS inequality when the nonnegative function $g$ satisfies that for all  decompositions $g=\phi+r$, there holds $\|r\|_{\frac{2n}{n+2s}}^2>\delta \|g\|_{\frac{2n}{n+2s}}^2$.
\begin{lemma}\label{nonlocal stability}
For fixed $s\in [1, \frac{n}{2})$, there exist $\delta\in(0,1)$, $n_0$ and $C_s$ which are dependent on $s$ such that for all $0 \leq g\in L^{\frac{2n}{n+2s}}(\mathbb{R}^n)$ with
\begin{align*}
\|r\|_{\frac{2n}{n+2s}}^2\geq \delta\|g\|_{\frac{2n}{n+2s}}^2,
\end{align*}
there holds
$$\left\| g \right\|_{\frac{2n}{n+2s}}^2-\mathcal S_{s,n} \|(-\Delta)^{-s/2}g\|_{2}^2\geq \frac{1}{n}C_s\inf_{h\in \mathcal{M}_{HLS}}\|g-h\|_{\frac{2n}{n+2s}}^2$$
when $n\geq n_0$.
\end{lemma}

\begin{proof}

Let $g_k=(\mathcal{R}U)^kg$ (see the definition of $\mathcal{R}U$ in Section 2). It is well known that $$k\rightarrow\|(-\Delta)^{-\frac{s}{2}}g_k\|^2_2$$
is increasing by the Riesz rearrangement inequality and $\|g_k\|_{\frac{2n}{n+2s}}=\|g\|_{\frac{2n}{n+2s}}$.
Thus
\begin{align}\label{est of hls}\nonumber
&\frac{\|g\|^2_{\frac{2n}{n+2s}}-S_{n,s}\|(-\Delta)^{-s/2}g\|_2^2}{\inf\limits_{h\in \mathcal{M}_{HLS}}\|g-h\|_{\frac{2n}{n+2s}}^2}\geq \frac{\|g\|^2_{\frac{2n}{n+2s}}-S_{n,s}\|(-\Delta)^{-s/2}g\|_2^2}{\|g\|^2_{\frac{2n}{n+2s}}}\\
&\geq  1-\frac{S_{n,s}\|(-\Delta)^{-s/2}g\|_2^2}{\|g\|^2_{\frac{2n}{n+2s}}}\geq \frac{\|g_k\|^2_{\frac{2n}{n+2s}}-S_{n,s}\|(-\Delta)^{-s/2}g_k\|_2^2}{\|g_k\|^2_{\frac{2n}{n+2s}}}.
\end{align}
Let $g_k=\phi_k+r_k$ with $\phi_k$ being any extremal of $\inf\limits_{h\in \mathcal{M}_{HLS}}\|(-\Delta)^{-\frac{s}{2}}(g_k-h)\|_{2}^2$.
From Lemma \ref{strong}, we know that
$$\lim\limits_{k\rightarrow +\infty}\|r_k\|_{\frac{2n}{n+2s}}=0.$$
Then it follows that
there exist a $k_0\in \mathbb{N}$ such that for all the decompositions $g_{k_0}=\phi_{k_0}+r_{k_0}$ of $g_{k_0}$, there must holds
\begin{equation}\label{uniform}
\|r_{k_0}\|_{\frac{2n}{n+2s}}>\delta\|g_{k_0}\|^2_{\frac{2n}{n+2s}}
\end{equation}
  and there exist one decomposition $g_{k_0+1}=\phi_{k_0+1}+r_{k_0+1}$ with
 $$\|r_{k_0+1}\|_{\frac{2n}{n+2s}}\leq \delta\|g_{k_0+1}\|^2_{\frac{2n}{n+2s}}.$$
 Denote $g_0=Ug_{k_0}$, $g_\infty=g_{k_0+1}$, then
 $$\inf\limits_{h\in \mathcal{M}_{HLS}}\|g_{0}-h\|^2_{\frac{2n}{n+2s}}=\inf\limits_{h\in \mathcal{M}_{HLS}}\|g_{k_0}-h\|^2_{\frac{2n}{n+2s}}
 >\delta\|g_{k_0}\|^2_{\frac{2n}{n+2s}}=\delta\|g_{0}\|^2_{\frac{2n}{n+2s}}.$$
 Now using the continuous rearrangement flow $g_\tau$ ($0\leq\tau\leq\infty$) introduced in Section 2, we conclude that $g_\tau$
 satisfies
$$\|(-\Delta)^{-\frac{s}{2}}g_\tau\|_{2}^2\geq \|(-\Delta)^{-\frac{s}{2}}g_0\|_{2}^2=\|(-\Delta)^{-\frac{s}{2}}g_{k_0}\|_{2}^2, ~\text{and}~\|g_\tau\|_{\frac{2n}{n+2s}}=\|g_0\|_{\frac{2n}{n+2s}}=\|g\|_{\frac{2n}{n+2s}}.$$

Now assume that $g_{\tau}=\phi_{\tau}+r_{\tau}$ with $\phi_{\tau}\in \mathcal{M}_{HLS}$ being an extremal of $\inf\limits_{h\in \mathcal{M}_{HLS}}\|g_\tau-h\|^2_{\frac{2n}{n+2s}}$ and define
$$\tau_0=\inf\{\tau\geq0:\text{there exist a decomposition}~~ g_\tau=\phi_\tau+r_\tau ~~\text{with}~~ \|r_\tau\|_{\frac{2n}{n+2s}}\leq \delta\|g_\tau\|_{\frac{2n}{n+2s}}\}.$$
If $\tau_0=0$, then there exists a sequence $\tau_k\rightarrow 0$ such that $g_{\tau_k}$ has a decomposition $g_{\tau_k}=\phi_{\tau_k}+r_{\tau_k}$ with
$$\|r_{\tau_k}\|_{\frac{2n}{n+2s}}\leq \delta\|g_{\tau_k}\|_{\frac{2n}{n+2s}}.$$
By Lemma \ref{continuity}, we know there exist a subsequence (still use the notation $\tau_k$)  $\tau_k\rightarrow 0$ such that $\phi_{\tau_k}\rightarrow \phi_0$, $r_{\tau_k}\rightarrow r_0$ in $L^{\frac{2n}{n+2s}}(\mathbb{R}^n)$
, $\phi_0$ being an extremal of $\inf\limits_{h\in M_{HLS}}\|g_0-h\|^2_{\frac{2n}{n+2s}}$ and $\|r_{0}\|_{\frac{2n}{n+2s}}\leq \delta\|g\|_{\frac{2n}{n+2s}}$, which is a contradiction with \eqref{uniform}.
\vskip 0.1cm

If $0<\tau_0\leq \infty$, by the definition of $\tau_0$ we know when $\tau<\tau_0$ for all the decompositions $g_\tau=\phi_\tau+r_\tau$, there must hold
$$\|r_\tau\|_{\frac{2n}{n+2s}}\geq \delta\|g_\tau\|_{\frac{2n}{n+2s}}.$$
Then by Lemma \ref{continuity} again, there is a subsequence $\tau_k\rightarrow \tau_0^{-}$ such that $g_{\tau_0}$ with a decomposition $g_{\tau_0}=\phi_{\tau_0^{-}}+r_{\tau_0^{-}}$ and $\|r_{\tau_0^{-}}\|_{\frac{2n}{n+2s}}\geq \delta\|g_{\tau_0}\|_{\frac{2n}{n+2s}}$.
At the same time by the definition of $\tau_0$ and Lemma \ref{continuity}, there is a subsequence $\tau_k^\prime\rightarrow \tau_0^{+}$ such that $g_{\tau_0}$ with another decomposition $g_{\tau_0}=\phi_{\tau_0^{+}}+r_{\tau_0^{+}}$ and
\begin{equation*}
\|r_{\tau_0^{+}}\|_{\frac{2n}{n+2s}}\leq \delta\|g_{\tau_0}\|_{\frac{2n}{n+2s}}.
\end{equation*}
Thus Lemma \ref{comparable} tells us
\begin{equation}\label{local}
\|r_{\tau_0^{+}}\|_{\frac{2n}{n+2s}}\geq b_{n,s}^{-1}\|r_{\tau_0^{-}}\|_{\frac{2n}{n+2s}}\geq b_{n,s}^{-1}\delta\|g_{\tau_0}\|_{\frac{2n}{n+2s}}
\end{equation}
since $b_{n,s}^{-1}\geq (1+2^{2+1/2}3^{-1/2})^{-1}$ when $n\geq 14s$.
Therefore by (\ref{est of hls}) and (\ref{local}), there holds
\begin{align*}
&\frac{\|g\|^2_{\frac{2n}{n+2s}}-S_{n,s}\|(-\Delta)^{-s/2}g\|_2^2}{\inf\limits_{h\in \mathcal{M}_{HLS}}\|g-h\|_{\frac{2n}{n+2s}}^2}\\
&\geq \frac{\|g_0\|^2_{\frac{2n}{n+2s}}-S_{n,s}^{-1}\|(-\Delta)^{-s/2}g_0\|_2^2}{\|g_0\|^2_{\frac{2n}{n+2s}}}\\
&\geq \frac{\|g_{\tau_0}\|^2_{\frac{2n}{n+2s}}-S_{n,s}^{-1}\|(-\Delta)^{-s/2}g_{\tau_0}\|_2^2}{\|g_{\tau_0}\|^2_{\frac{2n}{n+2s}}}\\
&\geq  (1+2^{2+1/2}3^{-1/2})^{-1}\delta \frac{\|g_{\tau_0}\|^2_{\frac{2n}{n+2s}}-S_{n,s}^{-1}\|(-\Delta)^{-s/2}g_{\tau_0}\|_2^2}{\|r_{\tau_0^{+}}\|_{\frac{2n}{n+2s}}^2}.
\end{align*}
Then by the local stability Lemma \ref{local stability}, we accomplish the proof of Lemma \ref{nonlocal stability}.

\end{proof}
To summarize, by Lemma \ref{local stability}, Lemma \ref{nonlocal stability} and Lemma \ref{constant} we have proved the stability of HLS inequalities with the optimal asymptotic lower bounds, namely we have completed the proof of Theorem \ref{sta of hls}.

\section{Stability of Sobolev inequalities with the optimal asymptotic lower bounds}
In this section we will set up the stability of the Sobolev inequalities with optimal asymptotic lower bounds from the stability of HLS inequality with optimal asymptotic lower bounds by adapting the strategy of the author's previous work~\cite{CLT1}. We give the details here for the sake of completeness.
\vskip 0.1cm

Let $f\in \dot{H}^s(\mathbb{R}^n)$ and denote $2^{\ast}_s=\frac{2n}{n-2s}$. Define $$\mathcal{F}(f)=\mathcal{S}_{s,n}^{-1}\|(-\Delta)^{s/2}f\|_2^2,~~~~\mathcal{E}(f)=\|f\|^2_{2^{\ast}_s}.$$
Then the Legendre transform $\mathcal{F}^{\ast}$ of a convex functional $\mathcal{F}: \dot{H}^s(\mathbb{R}^n)\rightarrow [0, +\infty)$ defined on $\dot{H}^{-s}(\mathbb{R}^n)$ is given by $$\mathcal{F}^{\ast}(g)=\sup_{f\in H^s(\mathbb{R}^n)}\{2\int_{\mathbb{R}^n}f(x)g(x)dx-\mathcal{F}(f)\}.$$
A simple calculation gives $\mathcal{F}^{\ast}(g)=\mathcal{S}_{s,n}\|(-\Delta)^{-s/2}g\|_2^2$. Similarly, the Legendre transform $\mathcal{E}^{\ast}$ of a convex functional $\mathcal{E}: L^{2^{*}_s}(\mathbb{R}^n)\rightarrow [0, +\infty)$ defined on $L^{\frac{2n}{n+2s}}(\mathbb{R}^n)$ is given by $$\mathcal{E}^{\ast}(g)=\sup_{f\in L^{2^{*}_s}(\mathbb{R}^n)}\{2\int_{\mathbb{R}^n}f(x)g(x)dx-\mathcal{E}(f)\}.$$
Obviously, $\mathcal{E}^{\ast}(g)=\|g\|^2_{\frac{2n}{n+2s}}$. Choose $g=\|f\|_{2^\ast}^{2-2^\ast_s}|f|^{2^\ast_s-1}sgn(f)$ and $f_1=\mathcal{S}_{s,n}(-\Delta)^{-s}g$, we can check that
\begin{align}\label{equation3}
\mathcal{E}(f)+\mathcal{E}^\ast(g)=2\int_{\mathbb{R}^n}fgdx
\end{align}
and
\begin{align}\label{equation4}\nonumber
&\mathcal{F}(f)=\mathcal{S}_{s,n}^{-1}\|(-\Delta)^{s/2}f\|_2^2\\\nonumber
& =\mathcal{S}_{s,n}^{-1}\|(-\Delta)^{s/2}f_1\|_2^2+\mathcal{S}_{s,n}^{-1}\|(-\Delta)^{s/2}(f-f_1)\|_2^2+2\int_{\mathbb{R}^n}(f-f_1)(\mathcal{S}_{s,n}^{-1}(-\Delta)^{s}f_1)dx\\\nonumber
& =\mathcal{S}_{s,n}\|(-\Delta)^{-s/2}g\|_2^2+2\int_{\mathbb{R}^n}(f-\mathcal{S}_{s,n}(-\Delta)^{-s}g)gdx+\mathcal{S}_{s,n}^{-1}\|(-\Delta)^{s/2}(f-f_1)\|_2^2\\
& =2\int_{\mathbb{R}^n}fgdx-\mathcal{F}^{\ast}(g)+\mathcal{S}_{s,n}^{-1}\|(-\Delta)^{s/2}(f-f_1)\|_2^2.
\end{align}
Combining (\ref{equation3}) with (\ref{equation4}), we have
\begin{align}\label{equation of deficit}
\mathcal{F}(f)-\mathcal{E}(f)= \mathcal{E}^\ast(g)-\mathcal{F}^\ast(g)+\mathcal{S}_{s,n}^{-1}\|(-\Delta)^{s/2}(f-f_1)\|_2^2.
\end{align}
Since we have already proved the stability of HLS inequality in Theorem \ref{sta of hls}
$$\mathcal{E}^\ast(g)-\mathcal{F}^\ast(g)\geq \frac{C_s}{n}\inf\limits_{h\in M_{HLS}}\|g-h\|^2_{\frac{2n}{n+2s}},$$
then for any $\epsilon>0$, there exists
a $g_0\in \mathcal{M}_{HLS}$ such that
\begin{align}\label{est of hls def}
\mathcal{E}^\ast(g)-\mathcal{F}^\ast(g)\geq \frac{C_s}{n}\|g-g_0\|^2_{\frac{2n}{n+2s}}-\varepsilon.
\end{align}
By (\ref{equation of deficit}), (\ref{est of hls def}), sharp HLS inequality and the fact $(-\Delta)^{-s/2}g_0\in \mathcal{M}_S$, we derive
\begin{align*}
& \mathcal{F}(f)-\mathcal{E}(f)\geq \frac{C_s}{n}\mathcal{S}_{s,n}\|(-\Delta)^{-s/2}(g-g_0)\|_2^2-\varepsilon+\mathcal{S}_{s,n}^{-1}\|(-\Delta)^{s/2}(f-f_1)\|_2^2\\
& =\frac{C_s}{n}\|\mathcal{S}_{s,n}^{1/2}(-\Delta)^{-s/2}(g-g_0)\|_2^2-\varepsilon+\|\mathcal{S}_{s,n}^{-1/2}(-\Delta)^{s/2}f-\mathcal{S}_{s,n}^{1/2}(-\Delta)^{-s/2}g\|_2^2\\
& \geq \frac{C_s}{2n}\|\mathcal{S}_{s,n}^{-1/2}(-\Delta)^{s/2}f-\mathcal{S}_{s,n}^{1/2}(-\Delta)^{-s/2}g_0\|_2^2-\varepsilon\\
& \geq \frac{\mathcal{S}_{s,n}^{-1}C_s}{2n}\inf_{h\in M_S}\|(-\Delta)^{s/2}(f-h)\|_2^2-\varepsilon,
\end{align*}
which means
$$\left\| (-\Delta)^{s/2} f \right\|_2^2-\mathcal S_{s,n} \|f\|_{\frac{2n}{n-2s}}^2\geq \frac{\beta_{s}}{n} \inf_{h\in\mathcal{ M}_s}\|(-\Delta)^{s/2}(f-h)\|_2^2.$$

\section{Appendix}
In this section, we will prove some lemmas which were used in our proof of the main Theorem.
\subsection{Decomposition}
First we prove the following decomposition lemma.

\begin{lemma}\label{decompostition}
For any $0 \leq g\in L^{\frac{2n}{n+2s}}(\mathbb{S}^n)$, there exists a $\phi\in M_{HLS} $ such that
$$\inf_{h\in M_{HLS}}\langle\mathcal{P}_{2s}(g-h),g-h\rangle=\langle\mathcal{P}_{2s}(g-\phi),g-\phi\rangle.$$
Moreover, denote $g=\phi+r$ and $\phi=c_{0}J_{\Phi}^{\frac{n+2s}{2n}}$, where $\Phi$ is a conformal transformation on the sphere, then
$$r\bot \text{span}\{J_{\Phi}^{1/2}Y_{0},J_{\Phi}^{1/2}Y_{1,i}\circ \Phi,~~i=1,\cdots,n+1\}.$$
\end{lemma}

\begin{proof}
First we claim that the infimum can be obtained. Denote by
$$F(c,\xi)=\langle\mathcal{P}_{2s}(g-v_{c,\xi}),g-v_{c,\xi}\rangle,$$ where $v_{c,\xi}=c\left(\frac{\sqrt{1+|\xi|^2}}{1-\xi\cdot \eta}\right)^{\frac{n+2s}{2}}$, then $F(c,\xi)$ is a continuous function on $\mathbb{R}\times B^n$, where
$B^{n+1}=\{x\in \mathbb{R}^{n+1}:|x|<1\}$. Assume $c_k\in \mathbb{R}$ and $\xi_{k}\in B^{n+1}$ such that
$$\lim_{k\rightarrow \infty}F(c_k,\xi_k)=\inf_{h\in M_{HLS}}\langle\mathcal{P}_{2s}(g-h),g-h\rangle.$$
Since
$$F(c_k,\xi_k)=(c_k-\langle\mathcal{P}_{2s}g,v_{c_k,\xi_k}\rangle)^2+\langle\mathcal{P}_{2s}(g),g\rangle-\langle\mathcal{P}_{2s}g,v_{c_k,\xi_k}\rangle^2,$$
and $\langle\mathcal{P}_{2s}g,v_{1,\xi_k}\rangle$ is bounded  by the Hardy-Littlewood-Sobolev inequality then we know
$|c_k|$ is bounded.
Then there exist $c_0\in \mathbb{R}$
and $\xi_0\in \mathbb{R}^{n+1}$ with $|\xi_0|\leq 1$ satisfying $(c_k,\xi_k)\rightarrow(c_0,\xi_0)$ (up to a subsequence). Now if $|\xi_0|<1$, by the continuity of $F(c,\xi)$ we obtain
$$F(c_0,\xi_0)=\inf_{h\in M_{HLS}}\langle\mathcal{P}_{2s}(g-h),g-h\rangle.$$
If $|\xi_0|=1$, by the Fatou lemma we have
\begin{align*}
& \langle\mathcal{P}_{2s}(g),g\rangle=\int_{\mathbb{S}^n}\liminf_{k\rightarrow \infty}\mathcal{P}_{2s}(g-c_k(\frac{\sqrt{1-|\xi_k|^2}}{1-\xi_k\cdot\omega})^{\frac{n-2s}{2}})(g-c_k(\frac{\sqrt{1-|\xi_k|^2}}{1-\xi_k\cdot\omega})^{\frac{n-2s}{2}})d\sigma_\xi\\
& \leq \liminf_{k\rightarrow \infty}\langle\mathcal{P}_{2s}(g-c_k(\frac{\sqrt{1-|\xi_k|^2}}{1-\xi_k\cdot\omega})^{\frac{n-2s}{2}}),g-c_k(\frac{\sqrt{1-|\xi_k|^2}}{1-\xi_k\cdot\omega})^{\frac{n-2s}{2}}\rangle\\
& =\inf_{h\in M_{HLS}}\langle\mathcal{P}_{2s}(g-h),g-h\rangle \leq \langle\mathcal{P}_{2s}(g),g\rangle,
\end{align*}
which means we can choose $c=0$ and any $|\xi|<1$ such that
$$\langle\mathcal{P}_{2s}(g-c(\frac{\sqrt{1-|\xi|^2}}{1-\xi\cdot\omega})^{\frac{n-2s}{2}}),g-c(\frac{\sqrt{1-|\xi|^2}}{1-\xi\cdot\omega})^{\frac{n-2s}{2}}\rangle
 =\inf_{h\in M_{HLS}}\langle\mathcal{P}_{2s}(g-h),g-h\rangle.$$
Then we complete the proof of the claim.
\vskip0.1cm

Next let $\phi=v_{c_{0},\xi_0}=c_0J_{\Phi}^{\frac{n+2s}{2n}}\in M_{HLS} $ such that
$$\inf_{h\in M_{HLS}}\langle\mathcal{P}_{2s}(g-h),g-h\rangle=\langle\mathcal{P}_{2s}(g-\phi),g-\phi\rangle,$$
then we have
$$r\bot \text{span}\{v_{1,\xi_0},\partial_\xi v_{c,\xi}|_{(c_0, \xi_0)} \}.$$
in the ``inner product" $\langle\mathcal{P}_{2s}\cdot,\cdot\rangle$. Next let us prove
$$\text{span}\{v_{1,\xi_0},\partial_\xi v_{c,\xi}|_{(c_0, \xi_0)} \}=\text{span} \{J_{\Phi}^{1/2}Y_{0,1}, J_{\Phi}^{1/2}Y_{1,i}\circ \Phi,~~i=1,\cdots,n+1\}.$$

Since $v_{c,\xi}$ is the extremal function of the
HLS, then it must satisfy the following Euler-Lagrange equation
$$\mathcal{P}_{2s}u=\|u\|_{\frac{2n}{n+2s}}^{\frac{4s}{n+2s}}|u|^{-\frac{4s}{n+2s}}u.$$
Differentiating at $(c_0,\xi_0)$, we know that $v_{1,\xi_0}$ and $\partial_{\xi_{i}} v_{c,\xi}|_{(c_0, \xi_0)}$ for $i=1, 2, \cdot\cdot\cdot, n+1$ satisfying the following equation (we may assume $c_0>0$)
\begin{align}\begin{split}\label{equation1}
& \mathcal{P}_{2s}u=\frac{4s}{n+2s}\|v_{c_0,\xi_0}\|_{\frac{2n}{n+2s}}^{\frac{4s-2n}{n+2s}}\int_{\mathbb{S}^n}v_{c_0,\xi_0}^{\frac{n-2s}{n+2s}}ud\sigma_\xi
v_{c_0,\xi_0}^{\frac{n-2s}{n+2s}}+\frac{n-2s}{n+2s}\|v_{c_0,\xi_0}\|_{\frac{2n}{n+2s}}^{\frac{4s}{n+2s}}v_{c_0,\xi_0}^{-\frac{4s}{n+2s}}u\\
&=\frac{4s}{n+2s}\int_{\mathbb{S}^n}J_\Phi^{\frac{n-2s}{2n}}ud\sigma_\xi J_\Phi^{\frac{n-2s}{2n}}+\frac{n-2s}{n+2s}J_\Phi^{-\frac{2s}{n}}u,
\end{split}\end{align}
which implies that $\text{span}\{v_{1,\xi_0},\partial_\xi v_{c,\xi}|_{(c_0, \xi_0)} \}$ belongs to the solution space of the
equation (\ref{equation1}).
We already know that $v_{c_0,\xi_0}=c_0J_{\Phi}^\frac{n+2s}{2n}$ for some conformal transformation $\Phi$ on $\mathbb{S}^n$. For any solution $u$ of the equation (\ref{equation1}), let $u_{\Phi^{-1}}=J_{\Phi^{-1}}^\frac{n+2s}{2n}u\circ\Phi^{-1}$. Then by the fact
$$\int_{\mathbb{S}^n}\frac{u_{\Phi^{-1}}(\xi)}{|\xi-\eta|^{n-2s}}d\sigma_\xi=J_{\Phi^{-1}}^{\frac{n-2s}{2n}}(\eta)\int_{\mathbb{S}^n}\frac{u(\xi)}{|\xi-\Phi^{-1}(\eta)|^{n-2s}}d\sigma_\xi,$$
we can get
$$\mathcal{P}_{2s}(u_{\Phi^{-1}})=J_{\Phi^{-1}}^{\frac{n-2s}{2n}}\mathcal{P}_{2s}u\circ\Phi^{-1}.$$

Using (\ref{equation1}) and $v_{c_0,\xi_0}=c_0J_{\Phi}^{\frac{n+2s}{2n}}$, we can obtain
\begin{align}\begin{split}\label{equation2}
& \mathcal{P}_{2s}(u_{\Phi^{-1}})=J_{\Phi^{-1}}^{\frac{n-2s}{2n}}\frac{4s}{n+2s}\int_{\mathbb{S}^n}J_\Phi^{\frac{n-2s}{2n}}ud\sigma_\xi J_\Phi^{\frac{n-2s}{2n}}(\Phi^{-1})+J_{\Phi^{-1}}^{\frac{n-2s}{2n}}\frac{n-2s}{n+2s}J_\Phi^{-\frac{2s}{n}}(\Phi^{-1})u(\Phi^{-1})\\
& =\frac{4s}{n+2s}\int_{\mathbb{S}^n}u_{\Phi^{-1}}d\sigma_\xi+\frac{n-2s}{n+2s}u_{\Phi^{-1}}
\end{split}\end{align}
where we use the fact $J_{\Phi^{-1}} J_\Phi(\Phi^{-1})=1$.
Expanding $u_{\Phi^{-1}} \in L^2(\mathbb{S}^n)$ in terms of the spherical harmonics,
\begin{equation}
\label{sph harm exp}\nonumber
u_{\Phi^{-1}}=\sum_{l=0}^{\infty}\sum_{m=1}^{N(n,l)} u_{l,m}Y_{l,m} \qquad \qquad \textrm{with} \qquad u_{l,m}=\int_{\mathbb{S}^n}u_{\Phi^{-1}}Y_{l,m} d \omega,
\end{equation}
by Funk-Hecke formula, we obtain
\begin{equation}\label{P2s} \mathcal{P}_{2s}(u_{\Phi^{-1}})=\sum_{l=0}^{\infty}\sum_{m=1}^{N(n,l)} u_{l,m}\frac{\Gamma(l+n/2-s)}{\Gamma(l+n/2+s)}Y_{l,m}.
\end{equation}
This together with (\ref{equation2}) and (\ref{P2s}) gives that
$$\sum_{l=0}^{\infty}\sum_{m=1}^{N(n,l)} u_{l,m}\frac{\Gamma(n/2+s)}{\Gamma(n/2-s)}\frac{\Gamma(l+n/2-s)}{\Gamma(l+n/2+s)}Y_{l,m}=\frac{4s}{n+2s}u_{0,1}+\frac{n-2s}{n+2s}\sum_{l=0}^{\infty}\sum_{m=1}^{N(n,l)} u_{l,m}Y_{l,m},$$
which implies $u_{l,m}=0$ when $l\geq 2$ since $\frac{\Gamma(n/2+s)}{\Gamma(n/2-s)}\frac{\Gamma(l+n/2-s)}{\Gamma(l+n/2+s)}<\frac{n-2s}{n+2s}$ when $l\geq 2$.
Thus we know that $u_{\Phi^{-1}}$ must be of the form  $$u_{\Phi^{-1}}=cY_{0,1}+\sum_{i=1}^{n+1}c_i Y_{1,i}.$$ So the
dimension of the solution space of equation (\ref{equation1}) is $n+1$. At the same time, the dimension of the space $\text{span}\{v_{1,\xi_0},\partial_\xi v_{c,\xi}|_{(c_0, \xi_0)} \}$ is also $n+1$, which completes the conclusion
$$\text{span}\{v_{1,\xi_0},\partial_\xi v_{c,\xi}|_{(c_0, \xi_0)} \}=\text{span} \{J_{\Phi}^{1/2}Y_{0,1}, J_{\Phi}^{1/2}Y_{1,i}\circ \Phi,~~i=1,\cdots,n+1\}.$$
\end{proof}

\subsection{Proof of Lemma \ref{q-estimate}} In our proof of local stability, Lemma \ref{q-estimate} plays an important role. Let us prove the lemma in this subsection and we need to prove the following
proposition and lemma first.
\begin{proposition}
Given $1<p_0<2$, $M>0$ and $N>\frac{p_0}{2(p_0-1)}$, there are positive constant $C_M$ and $C_{M,N}$ such that for any $0<\gamma\leq M$, $1<p_0\leq p\leq 2$ and $-1<r$, there holds
\begin{equation}\begin{split}\label{Elemetary inq1}
(1+r)^p-1-pr &\geq\frac 1 2p(p-1)(r_1+r_2)^2+2(r_1+r_2)r_3+(1-C_M N^{1-p_0}\ln N\theta)r_3^p\\
&\ \  -(\frac23\gamma\theta r_1^2+C_{M,N}\theta r_2^2)\chi_{\{r\leq M\}}-C_{M,N}\theta M^2\chi_{\{r>M\}}.
\end{split}\end{equation}
\end{proposition}

\begin{lemma}\label{fund ine}
let $1\leq p\leq 2$, then for all $r\geq -1$, there holds
$$(1+r)^p \geq 1+pr+\frac12p(p-1)r^2-\theta r_{+}^3.$$
\end{lemma}
\begin{proof}
For the case $-1\leq r\leq 0$, let $f(r)=(1+r)^p-1-pr-\frac12p(p-1)r^2$, then $f^{\prime\prime}(r)=p(p-1)[(1+r)^{p-2}-1]\geq 0$ since $p\leq 2$. Then $f^{\prime}(r)\leq f^{\prime}(0)=0$  and $f(r)\geq f(0)=0$ when
 $-1\leq r\leq 0$. For the case $r\geq 0$, using the fact $p\leq 2$ we have
 \begin{equation*}\begin{split}
(1+r)^p-1-pr-\frac12p(p-1)r^2&=p(p-1)(p-2)\int_0^r\int_0^s\int_0^t(1+x)^{p-3}dxdtds\\
&\ \  \geq p(p-1)(p-2)\int_0^r\int_0^s\int_0^tdxdtds\geq \theta r^3.
\end{split}\end{equation*}
\end{proof}
\begin{proof}[Proof of Proposition 1]
First we consider the case $r\leq M$. By Lemma \ref{fund ine} we have
$$(1+r)^p\geq 1+pr+\frac{1}{2}p(p-1)(r_1+r_2)^2-\theta(r_1+r_2)_{+}^3.$$
If $r\leq \gamma$, then $r_2=0$ and (\ref{Elemetary inq1}) follows from $(r_1)_{+}^3\leq \gamma r_1^2\leq \frac{3}{2}\gamma r_1^2$. If $\gamma\leq r \leq M$, then $r_1=\gamma$. Using the fact
$3r_1r_2\leq \frac12r_1^2+\frac92 r_2^2$, we have
$$(r_1+r_2)_{+}^3=\gamma r_1^2+3\gamma r_1r_2+3\gamma r_2^2+r_2^3\leq \frac32\gamma r_1^2+\frac{15}{2}\gamma+M)r_2^2,$$
which implies (\ref{Elemetary inq1}) with $C_{M,N}\geq \frac{17}{2}M$ since $\gamma\leq M$.

Then we consider the case $r>M$. Since $r=M+r_3$, then
$$(1+r)^p-1-pr=(1+r)^p-(1+r)^2+(1+M)^2-1-pM+\theta r_3+r_3^2+2Mr_3.$$
By the fact
$$(1+M)^2-1-pM-\frac{1}{2}M^2=\frac{1}{2}(2-p)(2+(p+1)M)\geq 0,$$
we have
\begin{equation}\begin{split}\label{est 1}
(1+r)^p-1-pr&\geq \frac12p(p-1)M^2+2Mr_3+r_3^2+(1+r)^p-(1+r)^2\\
&\ \   \geq \frac12p(p-1)(r_1+r_2)^2+2(r_1+r_2)r_3+r_3^2+(1+r)^p-(1+r)^2.
\end{split}\end{equation}
So we only need to bound the remaining term $r_3^2+(1+r)^p-(1+r)^2$. We do this separately in the subcases $M<r\leq M+N$ and $r>M+N$, where $N$ is an additional parameter.

If $M<r<M+N$, we have
$$(1+r)^p-(1+r)^2\geq -\theta (1+M+N)^2\ln (1+M+N)=-\theta C_{M,N}^1,$$
and $$r_3^2-r_3^p\geq -\theta N^2\ln N=-\theta C_{N}^1.$$
Inserting this into (\ref{est 1}) and choosing $C_{M,N}\geq M^{-2}(C_{M,N}^1+C_{M}^1)$, we have for $M<r<M+N$
$$(1+r)^p-1-pr\geq \frac12p(p-1)(r_1+r_2)^2+2(r_1+r_2)r_3+r_3^p-C_{M,N}\theta M^2\chi_{\{r>M\}},$$
which imply (\ref{Elemetary inq1}) in the subcase $M<r<M+N$.

At last, let handle the subcase $r\geq M+N$. Now we have $r_3=r-M>N$ and then by lemma (\ref{fund ine}) there holds
\begin{equation*}\begin{split}
(1+r)^p&=(1+M+r_3)^p=r_3^p(1+\frac{1+M}{r_3})^p \\
&\ \   \geq r_3^p+p(1+M)r_3^{p-1}+\frac12p(p-1)r_3^{p-2}(1+M)^2-\theta r_3^{p-3}(1+M)^3\\
&\ \   \geq r_3^p+p(1+M)r_3^{p-1}+\frac12p(p-1)r_3^{p-2}(1+M)^2-\theta N^{p-3}(1+M)^3\\
&\ \    =r_3^p+p(1+M)r_3^{p-1}+\frac12p(p-1)r_3^{p-2}(1+M)^2-\theta C_{M,N}^2,
\end{split}\end{equation*}
which, together with the fact
$$(1+r)^2=(1+M+r_3)^2=r_3^2+2r_3(1+M)+(1+M)^2$$
gives
\begin{equation*}\begin{split}
r_3^2+(1+r)^p-(1+r)^2&\geq r_3^p+(1+M)(pr_3^{p-1}-2r_3)\\
&\ \ +\left( \frac12p(p-1)r_3^{p-2}-1\right)(1+M)^2-\theta C_{M,N}^2.
\end{split}\end{equation*}
Let $h(t)=pt^{-1}-2t^{1-t}$ and $t_0=(\frac{p}{2(p-1)})^{\frac{1}{2-p}}$, then $h(t)$ is decreasing on $(0,t_0]$ and increasing on
$[t_0,\infty)$ and $t_0\leq \frac{p_0}{2(p_0-1)}$ since $1<p_0 \leq p\leq 2$. Then when $t\geq N>\frac{p_0}{2(p_0-1)}$, we get
$$pt^{p-1}-2t\geq (pN^{-1}-2N^{1-p})t^p.$$
On the other hand, since $1-N^{2-p}\geq N^{2-p}\ln N(p-2)\geq N^{2-p_0}\ln N(p-2)$, we have
$$pN^{-1}-2N^{1-p}=(p-2)N^{-1}+2N^{-1}(1-N^{2-p})\geq (p-2)N^{-1}(1+2N^{2-p_0}\ln N).$$
Thus when $t>N$
$$pt^{p-1}-2t\geq (pN^{-1}-2N^{1-p})t^p\geq -\theta N^{-1}(1+2N^{2-p_0}\ln N)t^p.$$

Let $g(t)=\frac12p(p-1)t^{-2}-t^{-p}$ and $t_1=(p-1)^{\frac{1}{2-p}}$, then it is easy to check that $h(t)$ is decreasing on $(0,t_1]$ and increasing on
$[t_1,\infty)$ and $t_1<1$. Then when $t\geq N>t_1$, we get
$$\frac12p(p-1)t^{p-2}-1\geq (\frac12p(p-1)N^{-2}-N^{-p})t^p.$$
Again using the fact $1-N^{2-p}\geq (p-2)$, we have
\begin{equation*}\begin{split}
\frac12p(p-1)N^{-2}-N^{-p}&\geq (\frac12p(p-1)-1)N^{-2}+N^{-2}-N^{-p}\\
&\ \ \geq \frac12(p-2)(p+1)N^{-2}+N^{-2}(p-2)N^{2-p_0}\ln N\\
&\ \ =-\theta N^{-2}(\frac{p+1}{2}+N^{2-p_0}\ln N).
\end{split}\end{equation*}
Thus when $t>N$,
$$\frac12p(p-1)t^{p-2}-1\geq -\theta N^{-2}(\frac{p+1}{2}+N^{2-p_0}\ln N)t^p.$$
Choose $C_M$ and $C_{M,N}^{(3)}$ such that for $N>\frac{2p_0}{p_0-1}$,
$$C_{M,N}^{(3)}=(1+M)N^{-1}(1+2N^{2-p_0}\ln N)+(1+M)^2N^{-2}(\frac{p+1}{2}+N^{2-p_0}\ln N\leq C_M N^{1-p_0}\ln N.$$
Then along with  we have
$$r_3^2+(1+r)^p-(1+r)^2\geq (1-\theta C_M N^{1-p_0}\ln N)r_3^p-\theta C_{M,N}^{(2)}.$$
Now let us choose $C_{M,N}=M^{-2}\max\{C_{M,N}^{(1)}+C_{M,N}^{(1)},C_{M,N}^{(2)},\frac{17}{2}M^3\}$, then we have complete the proof.

\end{proof}

Now we can prove Lemma \ref{q-estimate}.
Since $$p(p-1)r_1r_2=2r_1r_2-(3+\theta)\theta r_1r_2\geq 2r_1r_2-4\theta r_1r_2\geq 2r_1r_2-\frac{\gamma}{2}\theta r_1^2-\frac{8}{\gamma}\theta r_2^2,$$
and
$$C_{M,N}M^2\chi_{r>M}\leq C_{M,N}(M-\gamma)^2\chi_{r>M}\leq 4C_{M,N}r_2^2,$$
we deduce from Proposition 1 that
\begin{equation*}\begin{split}
(1+r)^p-1-pr &\geq \left(\frac12p(p-1)-2\gamma \theta\right)r_1^2+\left(\frac12p(p-1)-5C_{M,N}\theta-\frac{8}{\gamma}\theta\right )r_2^2\\
&\ \ +2r_1r_2+2(r_1+r_2)r_3+(1-C_M N^{1-p_0}\ln N\theta).
\end{split}\end{equation*}
Let us choose $N$ such that $N>\frac{p_0}{2(p_0-1)}$ and $C_M N^{1-p_0}\ln N\leq \varepsilon$. Then the proof is completed.

\bibliographystyle{amsalpha}

\end{document}